\documentclass[11pt]{amsart}

\usepackage{amsmath}
\usepackage{amsthm}
\usepackage{amssymb,stmaryrd}
\usepackage{enumerate}
\usepackage{pinlabel}
\usepackage{float}
\restylefloat{table}
\usepackage{longtable}

\theoremstyle{plain}
\newtheorem{thm}{Theorem}[section]

\newtheorem{cor}[thm]{Corollary}
\newtheorem*{coro}{Corollary}
\newtheorem{lemma}[thm]{Lemma}

\theoremstyle{definition}

\newtheorem{defn}[thm]{Definition}
\newtheorem{notation}[thm]{Notation}

\newtheorem{rem}[thm]{Remark}

\newcommand{\Sp}{\operatorname{Sp}}
\newcommand{\Mod}{\operatorname{Mod}}
\newcommand{\Fix}{\operatorname{Fix}}
\newcommand{\lcm}{\operatorname{lcm}}
\renewcommand{\O}{\operatorname{{\mathcal O}}}
\renewcommand{\o}{\operatorname{{\mathbb O}}}
\renewcommand{\S}{\operatorname{{\mathbb S}}}

\newcommand{\I}{\operatorname{{\mathcal I}}}

\newcommand{\p}{\operatorname{{\mathbb P}}}
\newcommand{\D}{\operatorname{{\mathcal D}}}
\newcommand{\E}{\operatorname{{\mathcal E}}}

\newcommand{\C}{\operatorname{{\mathcal{C}}}}
\newcommand{\csum}{\overline{\textbf{\Large \#}}}
\renewcommand{\l}{\operatorname{{\llbracket}}}
\renewcommand{\r}{\operatorname{{\rrbracket}}}
\newcommand\llbrace{\Big\{\!\big\{}
\newcommand\rrbrace{\big\}\!\Big\}}

\title{Roots of Dehn Twists about Multicurves}
\author{Kashyap Rajeevsarathy}
\address{Department of Mathematics\\
Indian Institute of Science Education and Research Bhopal\\
Bhopal Bypass Road, Bhauri \\
Bhopal 462 066, Madhya Pradesh\\
India}
\email{kashyap@iiserb.ac.in}
\urladdr{https://home.iiserb.ac.in/$_{\widetilde{\phantom{n}}}$kashyap/}

\author{Prahlad Vaidyanathan}
\address{Department of Mathematics\\
Indian Institute of Science Education and Research Bhopal\\
Bhopal Bypass Road, Bhauri \\
Bhopal 462 066, Madhya Pradesh\\
India}
\email{prahlad@iiserb.ac.in}
\urladdr{https://home.iiserb.ac.in/$_{\widetilde{\phantom{n}}}$prahlad/}

\subjclass[2000]{Primary 57M99; Secondary 57M60}

\keywords{surface, mapping class, Dehn twist, multicurve, root}

\begin{document}

\maketitle

%-----------------------------------
% Abstract
%-----------------------------------

\begin{abstract}
A \textit{multicurve} $\C$ on a closed orientable surface is defined
to be a finite collection of disjoint non-isotopic essential simple closed curves.
The Dehn twist $t_{\C}$ about $\C$ is the product of the Dehn twists
about the individual curves. In this paper, we give necessary and 
sufficient conditions for the existence of a root of such a Dehn twist, 
that is, a homeomorphism $h$ such that $h^n  = t_{\C}$. We give 
combinatorial data that corresponds to such roots, and use it to 
determine upper bounds for $n$. Finally, we classify all such roots
up to conjugacy for surfaces of genus 3 and 4. 
\end{abstract}

%-----------------------------------
% Introduction
%-----------------------------------

\section{Introduction}
For $g \geq 0$, let $S_g$ be the closed, orientable surface of genus $g$, and let $\text{Mod}(S_g)$ denote the mapping class group of $S_g$. By a \textit{multicurve} $\C$ in $S_g$, we mean a finite collection of disjoint non-isotopic essential simple closed curves in $S_g$. Let $t_c$ denote the left-handed Dehn twist about an essential simple closed curve $c$ on $S_g$. Since the Dehn twists about any two curves in $\C$ commute, we will define the \textit{left-handed Dehn twist about} $\C$ to be
$$
t_{\C} := \prod_{c\in \C} t_c
$$
A \textit{root of $t_{\C}$ of degree $n$} is an element $h \in \text{Mod}(S_g)$ such that $h^n = t_{\C}$. 

When $\C$ comprises a single nonseparating curve, D. Margalit and S. Schleimer \cite{MS} showed the existence of roots of $t_{\C}$ of degree $2g - 1$ in $\Mod(S_g)$, for $g \geq 2$. This motivated \cite{MK1}, in which D. McCullough and the first author derived necessary and sufficient conditions for the existence of a root of degree $n$. As immediate applications of the main theorem in the paper, they showed that roots of even degree cannot exist and that $n \leq 2g -1$. When $\C$ consists of a single separating curve,  the first author derived conditions \cite{KR1} for the existence of a root of $t_{\C}$. A stable quadratic upper bound on $n$, and complete classifications of roots for $S_2$ and $S_3$, were derived as corollaries 
to the main result. In this paper, we shall  derive conditions for the existence of a root of $t_{\C}$ when $|\C| \geq 2$, and since there are no such multicurves in $S_1$ or $S_0$, we shall assume henceforth that $g \geq 2$.

In general, a root $h$ of $t_{\C}$ may permute some curves in $\C$, while preserving other curves. So we define an \textit{$(r,k)$-permuting root} of $t_{\C}$ to be one that induces a partition of $\C$ into $r$ singletons and $k$ other subsets of size greater than one. The theory for $(r,0)$-permuting roots, as we will see, can be obtained by generalizing the theories developed in \cite{MK1} and \cite{KR1}, which involved the analysis of the fixed point data of finite cyclic actions.

The theory that we intend to develop for $(r,k)$-permuting roots when $k>0$ can be motivated by the following example. Consider the multicurve $\C$ in $S_5$ shown in Figure \ref{fig:perm_s5}. 

\begin{figure}[h!]
\labellist
\small
\pinlabel $2\pi/5$ at 30 200
%\pinlabel $h=r_{2\pi/5}t_{c_1}$ at 250 200
\endlabellist
\centering
\includegraphics[width=35 ex]{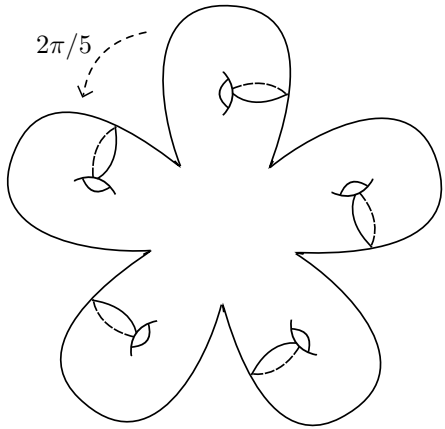}
\caption{Nonseparating multicurve of size 5 in $S_5$}\label{fig:perm_s5}
\end{figure}

\noindent It is apparent that the rotation of $S_5$ by $2\pi/5$ composed with $t_c$ for some fixed $c\in \C$ is a 5$^{th}$ root of $t_{\C}$ in $\Mod(S_5)$. This is a simple example of a $(0,1)$-permuting root, which is obtained by removing invariant disks around pairs of points in two distinct orbits of the $2\pi/5$ rotation of $S_0$, and then attaching five 1-handles with full twists. This example indicates that a classification of such roots would require the examination of the orbit information of finite cyclic actions, in addition to their fixed point data. This is a significant departure from existing theories developed in~\cite{MK1} and~\cite{KR1}. 

Any subset of a multicurve will be called a \textit{submulticurve}. A multicurve $\C$ in $S_g$ is said to be \textit{pseudo-nonseparating} if $\C$ separates $S_g$, but no proper submulticurve of $\C$ separates $S_g$. A multicurve that contains no pseudo-nonseparating submulticurves will be called a \textit{nonseparating multicurve}, while a multicurve which is a disjoint union of pseudo-nonseparating multicurves will be called a \textit{separating multicurve}. A multicurve that is neither separating nor nonseparating will be called a \textit{mixed multicurve}. In Figure~\ref{fig:s5mixed} below, the collection of curves $\C = \{c_1,c_2,c_3,c_4\}$ is a mixed multicurve, while the subcollections $\{c_2,c_3\}$, $\{c_1,c_2,c_3\}$ and $\{c_2,c_4\}$ form pseudo-nonseparating, separating and nonseparating multicurves, respectively.

\begin{figure}[h]

\labellist
\small
\pinlabel $c_1$ at 205 120
\pinlabel $c_2$ at 565 160
\pinlabel $c_3$ at 565 85
\pinlabel $c_4$ at 1053 85
\endlabellist
\centering
\includegraphics[width= 75 ex]{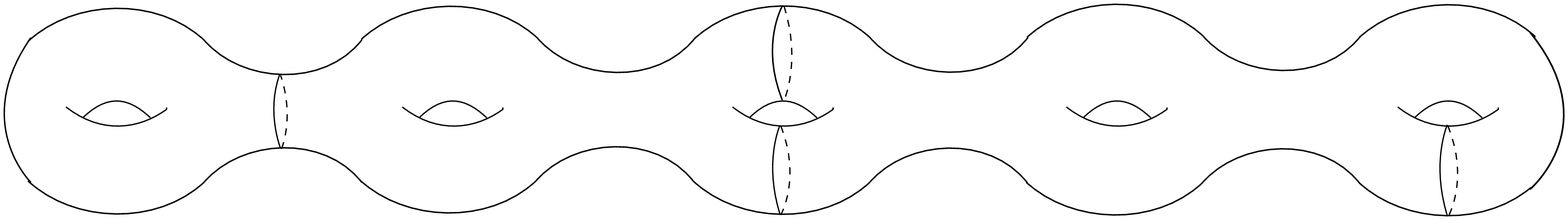}
\caption{The surface $S_5$ with a mixed multicurve}\label{fig:s5mixed}
\end{figure}

We start by generalizing the notion of a nestled $(n,\ell)$-action from~\cite{KR1} to a \textit{permuting $(n,r,k)$-action}. These are $C_n$-actions on $S_g$ that have $r$ distinguished fixed points, and $k$ distinguished non-trivial orbits. In Section~\ref{sec: actions_data_triples}, we introduce the notion of a \textit{permuting data set}, which is a generalization of a data set from \cite{KR1}. We use Thurston's orbifold  theory \cite[Chapter 13]{T1} in Theorem~\ref{thm:action_triple_correspondence} to establish a correspondence between permuting $(n,r,k)$-actions on $S_g$ and permuting data sets of genus $g$ and degree $n$. In other words, permuting data sets algebraically encode these permuting actions and contain all the relevant orbit and fixed-point information required to classify the roots that will be constructed from these actions.

Let $S_g(\C)$ denote the surface obtained from $S_g$ by deleting an annular neighbourhood of $\C$ and capping. In Section~\ref{sec:nonsepmulcurves}, we prove that conjugacy classes of roots of Dehn twists about nonseparating multicurves correspond to a special subclass of permuting actions on $S_g(\C)$. We use this to obtain the following bounds for the degree of such a root
\begin{coro}
Let $\C$ be a nonseparating multicurve in $S_g$ of size $m$, and let $h$ be an $(r,k)$-permuting root of $t_{\C}$ of degree $n$. 
\begin{enumerate}[(i)]
\item If $r\geq 0$, then
$$
n\leq \begin{cases}
4(g-m) + 2 &: g-m \geq 1 \\
g &: g = m.
\end{cases}
$$
Furthermore, if $g=m$, then this upper bound is realizable.
\item If $r\geq 1$, then $n$ is odd.
\item If $r=1$, then $n\leq 2(g-m)+1$. %Furthermore, if $g_{\C}\geq 2$, then $n \leq 2g_{\C}$
\item If $r\geq 2$, then $\displaystyle n \leq \frac{g-m+r-1}{r-1}$.
\end{enumerate}
\end{coro}

If $\C$ is a separating or a mixed multicurve, $S_g(\C)$ will have multiple connected components. In order to classify roots of $t_{\C}$ in this case, we will require multiple actions on the various components of $S_g(\C)$ which can be put together and extended to a root of $t_{\C}$ on $S_g$. In Section~\ref{sec:sepmulcurves}, we will show that this extension will require compatibility of orbits across components of $S_g(\C)$ in addition to compatibility of fixed points as in \cite{KR1}. As an immediate consequence to this theorem, we obtain quadratic bounds for the degree of the root in terms of the genera of the individual components. Furthermore, we obtain a quadratic stable upper bound for the degree of the root as in \cite[Theorem 8.14]{KR1}.

In Sections \ref{sec:classify_genus3} and \ref{sec:classify_genus4}, we use our theory to obtain a complete classification of roots of $t_{\C}$ on surfaces of genus 3 and 4 respectively. We conclude by proving that a root of $t_{\C}$ cannot lie in the Torelli group of $S_g$, and also indicate how our results can be extended to classify roots of finite products of powers of commuting Dehn twists.
integrated 
%-----------------------------------
% Permuting actions and Data sets
%-----------------------------------

\section{Roots and their induced partitions}
In this section we shall introduce some preliminary notions, which will be used in later sections.

\begin{notation}\label{defn:sgc}
Let $\C$ be a multicurve in $S_g$, and let $N$ be a closed annular neighbourhood of $\C$. 
\begin{enumerate}[(i)]
\item We denote the surface $\overline{S_g\setminus N}$ by $\widehat{S_g(\C)}$. 
\item The closed orientable surface obtained from $\widehat{S_g(\C)}$ by capping off its boundary components is denoted by $S_g(\C)$. 
\item If $\C$ is a nonseparating multicurve, then $S_g(\C)$ is a connected surface whose genus we denote by $g_{\C}$.
\end{enumerate}
%\noindent Note that if $\C$ contains a separating multicurve, then $S_g(\C)$ is disconnected.
\end{notation}

\begin{notation}\label{defn:pseudo_multicurve}
Recall that a multicurve $\C$ in $S_g$ is said to be \emph{pseudo-nonseparating} if $S_g(\C)$ is disconnected, but $S_g(\C')$ is not disconnected for any proper submulticurve $\C'\subset \C$.
\begin{enumerate}[(i)]
\item  If $|\C| = k$, we write $\C^{(k)}$ for such a multicurve. Note that $\C^{(1)}$ is a single separating curve. 
\item A disjoint union of $m$ copies of $\C^{(k)}$ is denoted by $\C^{(k)}(m)$.
\item For integers $g \geq 0$ and $m \geq 1$, we define $\S_g(m)$ to be the disjoint union of $m$ copies $\{S_g^1, S_g^2, \ldots, S_g^m\}$ of $S_g$ isometrically imbedded in $\mathbb{R}^3$. In particular, $\S_g(1) \cong S_g$, and hence we shall write $S_g$ for $\S_g(1)$. 
\item Given two surfaces $S_{g_1}$ and $\S_{g_2}(m)$ and a fixed $k\in \mathbb{N}$, we construct a new surface $S_g$ with $g=(g_1+mg_2 + (k-1)m)$, containing a multicurve of type $\C^{(k)}(m)$, in the following manner. We remove $km$ disks $\{D_{i,j}^1 : 1\leq j\leq k, 1\leq i\leq m\}$ on $S_{g_1}$ and $k$ disks $\{D_{i,j}^2 : 1\leq j\leq k\}$ on each $S_{g_2}^i$. Now connect $\partial D_{i,j}^1$ to $\partial D_{i,j}^2$ along a 1-handle $A_{i,j}$, and choose the unique curve (upto isotopy) $c_{i,j}$ on $A_{i,j}$. Let $\C = \{c_{i,j}\}$, then note that $\C = \C^{(k)}(m)$, so we write
$$
S_{g_1}\#_{\C} \S_{g_2}(m)
$$
for the new surface $S_g$.
\item Similarly, given surfaces $\{S_{g_1}, \S_{g_2,1}(m_1), \ldots, \S_{g_2,s}(m_s)\}$ and non-negative integers $\{k_1,k_2,\ldots, k_s\}$, we construct a new surface $S_g$ with $g= g_1 + \sum_{i=1}^s m_i(g_{2,i} + k_i-1)$, containing a multicurve of type $\C = \sqcup_{i=1}^s \C^{(k_i)}(m_i)$. Let
$$
S_{g_i'} := S_{g_1}\#_{\C^{(k_i)}(m_i)} \S_{g_2,i}(m_i)
$$
and $\C_i := \C\setminus \C^{(k_i)}(m_i)$, we now define
$$
 \csum_{i=1}^s \left( S_{g_1}{\#_{\C^{(k_i)}(m_i)}} \S_{g_2,i}(m_i) \right) := \bigcup_{i=1}^s \widehat{S_{g_i'}(\C_i)}.
$$ 
If $s = 2$, we simply write $S_g = \S_{g_{2,1}}(m_1)\#_{\C^{(k_1)}(m_1)} S_{g_1}\#_{\C^{(k_2)}(m_2)} \S_{g_{2,2}}(m_2)$.
%We will see later that such constructions will play a crucial role roots of $t_{\C}$ where $\C = \sqcup_{i=1}^s \C^{(k_i)}(m_i)$ on $S_g$.
\end{enumerate}
\end{notation}
\noindent In Figure~\ref{fig:sg_sgm} below, we give an example of a such a surface $S_{22}$ with a multicurve $\C = \C^{(2)}(2) \sqcup \C^{(1)}(3)$ constructed as in Notation~\ref{defn:pseudo_multicurve} from the surfaces $S_5, \S_3(2)$, and $\S_3(3)$. 
\begin{figure}[h]
\labellist
\tiny
\pinlabel $\C^{(2)}(2)$ at 70 95
\pinlabel $\C^{(1)}(3)$ at 620 35
%\pinlabel $c_2'$ at 565 160
%\pinlabel $c_2''$ at 565 85
%\pinlabel $c_3$ at 1053 85
\endlabellist
\centering
\includegraphics[width= 70 ex]{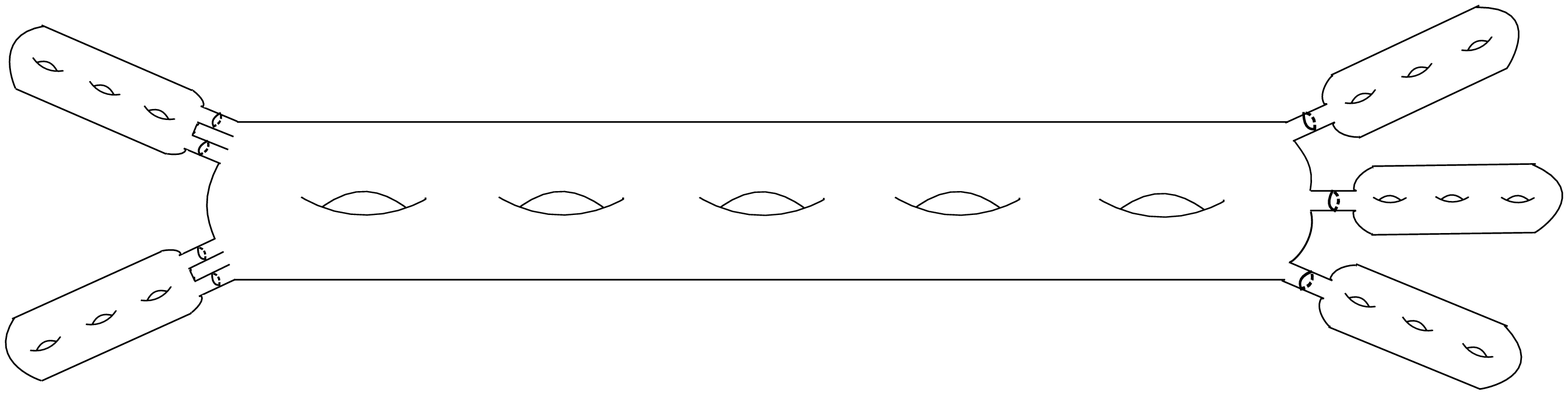}
\caption{The surface $S_{22} = \S_3(2) \#_{\C_1} S_5 \#_{\C_2} \S_3(3),$ where $\C_1 = \C^{(2)}(2)$ and $\C_2 = \C^{(1)}(3)$.}
\label{fig:sg_sgm}
\end{figure}

\begin{rem}\label{rem:root_isotopy}
Let $\C = \{c_1,c_2,\ldots, c_m\}$, and suppose that $h$ is a root of $t_{\C}$ of degree $n$ in $\Mod(S_g)$. 
Then, we claim that, up to isotopy, $h(\C) = \C$.

Suppose first that $h(c_i) \nsim c_1$ for all $i$. Then there exists a neighbourhood $N$ of $c_1$ such that
$$
t_{h(c_i)} \mid_N = \text{id}_N \text{ and } t_{c_i}\mid_N = \text{id}_N \quad\forall i\neq 1
$$
However, since $h^n = t_{\C}$, it follows that
$$
t_{\C} = ht_{\C}h^{-1} = ht_{c_1}h^{-1}ht_{c_2}h^{-1}\ldots ht_{c_m}h^{-1} = t_{h(c_1)}t_{h(c_2)}\ldots t_{h(c_m)}
$$
Hence
$$
t_{c_1}\mid_N = t_{\C}\mid_N = t_{h(c_1)}t_{h(c_2)}\ldots t_{h(c_m)}\mid_N = \text{id}_N
$$
which is a contradiction.

So we assume without loss of generality that $h(c_1) \sim c_1$. Then we choose a 
neighbourhood $N'$ of $c_1$ disjoint from $c_i$ for $i\neq 1$, and choose an isotopy $\varphi_t$ such that
$$
\varphi_t\mid_{S_g\setminus N'} = \text{id}_{S_g\setminus N'}
$$
and $\varphi_0 = h$ and $\varphi_1(c_1) = c_1$. Replacing $h$ by $\varphi_1$, we may assume that $h(c_1) = c_1$. Now note that $h(c_i) \nsim c_1$ for all $i\neq 1$, which allows us to proceed 
by induction on $|\C|$ to conclude that, up to isotopy, $h(\C) = \C$.
\end{rem}

\begin{defn}\label{defn:rkpartition}
Let $\C$ be a multicurve of size $m$ in $S_g$. Then for integers $r,k \geq 0$, an \textit{$(r,k)$-partition} 
of $\C$ is a partition $\p_{r,k}(\C) = \{\C_1',\ldots,\C_r',\C_1,\ldots,\C_k\}$ of the set $\C$ into subsets such that for all $i$, 
\begin{enumerate}[(i)]
\item $|\C_i'| =1$, $|\C_i| > 1$, and 
\item $\C_i$ comprises only separating or only nonseparating curves.
\end{enumerate}
\end{defn}

\begin{defn}
Let $\C$ be a multicurve in $S_g$. Then for integers $r,k \geq 0$, a root $h$ of $t_{\C}$ of is said to be \textit{$(r,k)$-permuting} if it induces an $(r,k)$-partition of $\C$. 
\end{defn}

\begin{notation}\label{notation:rkpartition}
Let $\C$ be a multicurve of size $m$ in $S_g$ and consider an $(r,k)$-partition $\p_{r,k}(\C) = \{\C_1',\ldots,\C_r',\C_1,\ldots, \C_k\}$ as in Definition~\ref{defn:rkpartition}. 
\begin{enumerate}[(i)]
\item We shall denote the multiset $\llbrace |\C_1|,|\C_2|, \ldots,|\C_k|\rrbrace$ by $S(\p_{r,k})$. (From here on, we shall denote a multiset using $\llbrace\,\,\rrbrace$).
\item If $\p_{r,k}(\C)$ is induced by an $(r,k)$-permuting root $h$ of $t_{\C}$, then we shall denote it by $\C_{r,k}(h)$.
\end{enumerate}
\end{notation}

\section{Permuting Actions and Permuting Data Sets}\label{sec: actions_data_triples}
In this section we shall introduce permuting $(n,r,k)$-actions, which are generalizations of the nestled $(n,\ell)$-actions from~\cite{KR1}. We shall also introduce the notion of a permuting $(n,r,k)$-data set, which is an abstract tuple involving non-negative integers that would algebraically encode a permuting $(n,r,k)$-action. 
\begin{defn}\label{defn:permuting_actions}
For integers $n\geq 1 \text{ and }r,k\geq 0$, an orientation-preserving $C_n$-action $t$ on $S_g$ is called a \emph{permuting $(n,r,k)$-action} if
\begin{enumerate}[(i)]
\item there is a set $\p(t) \subset S_g$ of $r$ distinguished fixed points of $t$, and 
\item there is a set $\o(t) \subset S_g$ of $k$ distinguished non-trivial orbits of $t$.
\end{enumerate}
%Note that if $t$ is a free action, then $r=0$ and $|\o| = n$ for all $\o \in \o_t$.
\end{defn}
\begin{notation}\label{defn:orbit_distribution}
Let $t$ be a permuting $(n,r,k)$-action on $S_g$. 
\begin{enumerate}[(i)]
\item Fix a point $P \in S_g$, and consider $t_{\ast}: T_P(S_g) \to T_{t(P)}(S_g)$.  By the Nielsen realisation theorem~\cite{K1}, we may change $t$ by isotopy in $\Mod(S_g)$ so that $t_{\ast}$ is an isometry. Hence, $t_{\ast}$ induces a local rotation by an angle, which we shall denote by $\theta_P(t)$. Note that if $P \in \p(t)$, then $\theta_P(t) = 2\pi a/n,$ where $\gcd(a,n) = 1$. 

%\item Fix a point $P \in \p(t)$, and consider $t_{\ast}: T_P(S_g) \to T_{P}(S_g)$.  By the Nielsen realisation theorem~\cite{K1}, we may change $t$ by isotopy in $\Mod(S_g)$ so that $t_{\ast}$ is an isometry. Hence, $t_{\ast}$ induces a local rotation by an angle, which we shall denote by $\theta_P(t)$. Note that $\theta_P(t) = 2\pi a/n,$ where $\gcd(a,n) = 1$. 

\item Fix an orbit $\o=\{P_1,\ldots,P_s\} \in \o(t)$. If $s<n$, then $s\mid n$, and there exists a cone point in the quotient orbifold of degree $n/s$. Each $P_i$ has stabilizer generated by $t^s$ and the rotation induced by $t^s$ around each $P_i$ must be the same, since its action at one point is conjugate by a power of $t$ to its action at each other point in the orbit. So the rotation angle is of the form $2\pi c^{-1}/(n/s)\pmod{2\pi}$, where $(c,n/s) = 1$ and $c^{-1}$ denotes the inverse of $c\pmod{n/s}$. We now associate to this orbit a pair $p(\o)$ as follows: 
$$
p(\o) := \begin{cases}
(c, n/s) &\text{if } s<n \\
(0,1) &\text{if } s=n.
\end{cases}
$$

\item For any orbit $\o \in \o(t)$, if $p(\o) = (a,b)$, then we define
$$
\theta_{\o}(t) := 
\begin{cases}
2\pi a^{-1}/b &\text{if } a\neq 0 \\
0 &\text{if }a=0.
\end{cases}
$$
\end{enumerate}
\end{notation}

\begin{defn}\label{defn: orbit_distribution}
 Consider a permuting $(n,r,k)$-action $t$ on $S_g$ with $\p(t) = \{P_1,\ldots,P_r\}$ and $\o(t) = \{\o_1,\o_2,\ldots, \o_k\}$.
\begin{enumerate}[(i)]
\item We write
$$
S(t) = \llbrace |\o_1|, |\o_2|, \ldots, |\o_k|\rrbrace.
$$
\item  For each $p \in \{p(\o_i) : 1\leq i \leq k\}$, define
$$
m_p = |\{j : p(\o_j) = x\}|.
$$
We define the \emph{orbit distribution} of $t$ to be the set
$$
\o_t = \{(p,m_p) :  \in \{p(\o_i): 1\leq i\leq k\}\}.
$$
\end{enumerate}
\end{defn}

\begin{defn}
\label{defn:eq_perm_actions}
Let $t_1$ and $t_2$ be two permuting $(n,r,k)$-actions on $S_g$ with $\p(t_s) = \{P_{s,1},P_{s,2},\ldots,P_{s,r}\}$ and $\o(t_s) = \{\o_{s,1},\o_{s,2},\ldots, \o_{s,k}\}$ for $s=1,2$. We say $t_1$ is \emph{equivalent} to $t_2$ if $\o_{t_1} = \o_{t_2}$ and there is an orientation-preserving homeomorphism $\phi \in \Mod(S_g)$ such that 
\begin{enumerate}[(i)]
\item $\phi(P_{1,i}) = P_{2,i}$ for $1 \leq i \leq r$,
\item for each $1\leq j\leq k$, if $\o_{s,j} = \{Q_{j,1}^s, Q_{j,2}^s, \ldots, Q_{j,m_{s_j}}^s\}$, then $m_{1_j} = m_{2_j}$ and $\phi(Q_{j,i}^1) = Q_{j,i}^2$ for all $1\leq i\leq m_{1_j}$, and
\item $\phi t_1 \phi^{-1}$ is isotopic to $t_2$ relative to $\p(t_2)\sqcup(\cup_{j=1}^k \o_{2,j})$.
\end{enumerate}
The equivalence class of a permuting $(n,r,k)$-action is denoted by $\l t\r$.
\end{defn}

We now introduce the notion of an \textit{$(n,r)$-data set}, which encodes the signature of the quotient orbifold of a permuting $(n,r,k)$-action and the turning angles around its distinguished fixed points. Furthermore, the $(n,r)$-data set will be combined with the orbit distribution of the action to form a pair, which we will call a \textit{permuting $(n,r,k)$-data set}.
\begin{defn}\label{defn:data_sets}
Given $n\geq 1$ and $r\geq 0$, an \textit{$(n,r)$-data set} is a tuple
$$
\D = (n,g_0, (a_1,a_2,\ldots, a_r); (c_1,n_1), (c_2,n_2),\ldots, (c_s,n_s))
$$
where $n\geq 1$ and $ g_0 \geq 0$ are integers,  each $a_i$ is a residue class modulo $n$, and each $c_i$ is a residue class modulo $n_i$ such that:
\begin{enumerate}[(i)]
\item each $n_i\mid n$,
\item $gcd(a_i,n) = gcd(c_i,n_i) = 1$, and
\item $\displaystyle \sum_{i=1}^r a_i + \sum_{j=1}^s \frac{n}{n_i}c_i \equiv 0\pmod{n}$.
\end{enumerate}
The number $g$ determined by the equation
\begin{equation}\label{eqn:riemann_hurwitz}
\frac{2-2g}{n} = 2-2g_0 + r\left(\frac{1}{n}-1 \right) + \sum_{j=1}^s \left(\frac{1}{n_j} - 1 \right)
\end{equation}
is called the \emph{genus} of the data set.
\end{defn}

\begin{defn} Fix an $(n,r)$-data set $\D$ of genus $g$ as above.
\begin{enumerate}[(i)]
\item For each $(a,b) \in \{(0,1),(c_1,n_1), \ldots, (c_s,n_s)\}$, 
we write
$$
\theta((a,b)) := \begin{cases}
0 & \text{if }a=0, \\
2\pi a^{-1}/b &\text{if } a\neq 0.
\end{cases}
$$
\item For each $p \in \{(0,1),(c_1,n_1), \ldots, (c_s,n_s)\}$, choose a non-negative integer $m_p$. Then the set $\o_{\D} = \{(p,m_p) : m_p > 0\}$
is called an \emph{orbit distribution} of $\D$.
\item Given an orbit distribution $\o_{\D}$ associated with an $(n,r)$-data set $\D$, the pair $(\D, \o_{\D})$ is called a \emph{permuting $(n,r,k)$-data set} of genus $g$, where $k=\sum_p m_p$.
\end{enumerate}
\end{defn}

\begin{defn} Let 
\begin{equation*}
\begin{split}
\D &= (n,g_0, (a_1,a_2,\ldots, a_r); (c_1,n_1), (c_2,n_2),\ldots, (c_s,n_s)) \\
\text{ and } \D' &= (n,g_0', (a_1',a_2',\ldots, a_r'); (c_1',n_1'), (c_2',n_2'),\ldots, (c_s',n_s'))
\end{split}
\end{equation*}
be two $(n,r)$-data sets as in Definition \ref{defn:data_sets}.
\begin{enumerate}[(i)]
\item $\D$ and $\D'$ are said to be \emph{equivalent} if 
$$
\llbrace a_1,a_2,\ldots, a_r\rrbrace = \llbrace a_1', a_2', \ldots, a_r'\rrbrace, \text{ and }
$$
$$
\llbrace (c_1,n_1),\ldots, (c_s,n_s)\rrbrace = \llbrace (c_1',n_1'), \ldots (c_s',n_s')\rrbrace.
$$
\item Two permuting $(n,r,k)$-data sets $(\D, \o_{\D})$ and $(\D', \o_{\D'})$ are said to be \emph{equivalent} if $\D$ and $\D'$ are equivalent as above, and $\o_{\D} = \o_{\D'}$.
\end{enumerate}
\end{defn}
\noindent Note that equivalent data sets have the same genus. 

\begin{thm}\label{thm:action_triple_correspondence}
Given $n\geq 1$ and $g\geq 0$, equivalence classes of permuting $(n,r,k)$-data sets of genus $g$ correspond to equivalence classes of permuting $(n,r,k)$-actions on $S_g$.
\end{thm}
\begin{proof} 
Let $t$ be a permuting $(n,r,k)$-action on $S_g$ with quotient orbifold $\O$ whose underlying surface has genus $g_0$. If $t$ is a free action, then $\O = S_{g_0}$, and we simply write $\D = (n,g_0;)$ and $\o_{\D} = \{(0,1),k)\}$. If $t$ is not free, let $p_j$ be the image in $\O$ of the $P_j$, for $1 \leq j\leq r$, and let $q_1,q_2,\ldots, q_s$ be the other possible cone points of $\O$ as in Figure \ref{fig:orb}.

\begin{figure}[h]
\label{fig:orb}
\labellist
\small
\pinlabel $p_1$ at 810 50
\pinlabel $p_2$ at 785 91
\pinlabel $q_1$ at 686 103
\pinlabel $q_2$ at 645 75
\endlabellist
\centering
\includegraphics[width = 75 ex]{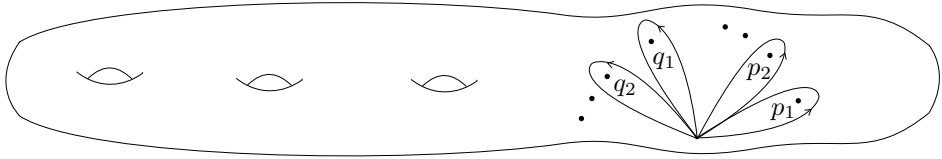}
\caption{The quotient orbifold $\O$.}
\end{figure}

Let $\alpha_i$ be the generator of the orbifold fundamental group $\pi_1^{orb}(\O)$ that goes around the point $p_i, 1\leq i\leq r$ and let $\gamma_j$ be the generators going around $q_j,1\leq j\leq s$. Let $x_p$ and $y_p,1\leq p\leq g_0$ be the standard generators of the ``surface part'' of $\O$, chosen to give the following presentation of $\pi_1^{orb}(\O)$:
\begin{gather*}
\pi_1^{orb}(\O) = \langle \alpha_1,\alpha_2,\ldots, \alpha_r, \gamma_1,\gamma_2,\ldots, \gamma_s,x_1,y_1,x_2,y_2,\ldots x_{g_0},y_{g_0} \lvert \\ \alpha_1^n = \alpha_2^n = \ldots = \gamma_1^{n_1} = \gamma_2^{n_2} = \ldots = \gamma_s^{n_s} = 1, \alpha_1\ldots \alpha_r\gamma_1\ldots \gamma_s = \prod_{p=1}^{g_0}[x_p,y_p]\rangle
\end{gather*}
From orbifold covering space theory \cite{T1}, we have the following exact sequence
$$
1 \rightarrow \pi_1(S_g) \rightarrow \pi_1^{orb}(\O) \xrightarrow{\rho} C_n\rightarrow 1,
$$
where $C_n = \langle t \rangle$. 
The homomorphism $\rho$ is obtained by lifting path representatives of elements of $\pi_1^{orb}(\O)$. Since these do not pass through the cone points, the lifts are uniquely determined. 

For $1\leq i\leq s$, the preimage of $q_i$ consists of $n/n_i$ points cyclically permuted by $t$. As in Notation~\ref{defn:orbit_distribution}, the rotation angle at each point is of the form $2\pi c_i^{-1}/n_i$ where $c_i$ is a residue class modulo $n_i$ and $gcd(c_i,n_i) = 1$. Lifting the $\gamma_i$, we have that $\rho(\gamma_i) = h^{(n/n_i)c_i}$. Similarly, lifting the $\alpha_i$ gives $\rho(\alpha_i)=t^{a_i}$ where $gcd(a_i,n)=1$. Finally, we have
$$
\rho(\prod_{p=1}^{g_0}[x_p,y_p]) = 1,
$$
since $C_n$ is abelian, so
$$
1 = \rho(\alpha_1\ldots \alpha_r\gamma_1\ldots \gamma_s) = t^{a_1+\ldots + a_r+(n/n_1)c_1 + \ldots + (n/n_s)c_s}
$$
giving
$$
\sum_{i=1}^r a_i + \sum_{j=1}^s \frac{n}{n_j}c_j \equiv 0 \pmod{n}.
$$
The fact that the data set $\D$ has genus $g$ follows easily from the multiplicativity of the orbifold Euler characteristic for the orbifold covering $S_g \to \O$:
$$
\frac{2-2g}{n} = 2-2g_0 + r\left(\frac{1}{n}-1 \right) + \sum_{j=1}^s \left(\frac{1}{n_j} - 1 \right).
$$
Thus, $h$ gives a $(n,r)$-data set
$$
\D = (n,g_0, (a_1,a_2,\ldots, a_r);(c_1,n_1),(c_2,n_2),\ldots, (c_s,n_s))
$$
of genus $g$, and hence $(\D, \o_t)$ forms a permuting $(n,r,k)$-data set. 

Consider another permuting $(n,r,k)$-action $t'$ in the equivalence class of
$t$ with a distinguished fixed point set $\p(t') = \{P_1',P_2',\ldots,P_r'\}$. Then by definition there exists an orientation-preserving homeomorphism $\phi \in \Mod(S_g)$
such that $\phi(P_j) = P_j'$ for all $j$ and $\phi t\phi^{-1}$ is isotopic to $t'$ relative to $\p(t')$. Therefore, $\theta_{P_j}(t) = \theta_{P_j'}(t')$, for $1\leq j\leq r$, and since $\o_t = \o_{t'}$, the two actions will produce the same permuting $(n,r,k)$-data sets.

Conversely, given a permuting $(n,r,k)$-data set $(\D, \o_{\D})$, we construct the orbifold $\O$ and a representation $\rho : \pi_1^{orb}(\O) \to C_n$. Any finite subgroup of $\pi_1^{orb}(\O)$ is conjugate to one of the cyclic subgroups generated by $\alpha_j$ or a $\gamma_i$, so condition (iv) in the definition of the data set ensures that the kernel of $\rho$ is torsion-free. Therefore, the orbifold covering $S \to \O$ corresponding to the kernel is a manifold, and calculation of the Euler characteristic shows that $S = S_g$. Thus we obtain a $C_n$-action $t$ on $S_g$ with $r$ distinguished fixed points $\p(t)$. We now construct $\o_t$ from $\o_{\D}$ in the following manner. For each pair $(p,m_p) \in \o_{\D}$, write $p = (a,b)$. If $a=0$, then choose $m_p$ orbits of size $n$ (if $t$ is a free action, this choice is trivial, but otherwise, such an orbit would always exist in a small neighbourhood around any fixed point of $t$). If $a\neq 0$, then there exists a cone point in $\O$ of degree $b$, so there exists an orbit of $t$ of size $n/b$ in $S_g$. Once again, by considering a small neighbourhood of this orbit, we may choose $m_p$ distinct orbits $\{\o_p^1, \o_p^2, \ldots, \o_p^{m_p}\}$ and set
$$
\o(t) := \bigsqcup_{(p,m_p) \in \o_{\D}} \{\o_p^1, \o_p^2, \ldots, \o_p^{m_p}\},
$$
which in turn gives $\o_t = \o_{\D}$.

It remains to show that the resulting action on $S_g$ is determined upto our equivalence in $\Mod(S_g)$. Suppose that two permuting $(n,r,k)$-actions $t$ and $t'$ have the same permuting $(n,r,k)$-data set $(\D,\o_{\D})$. $\D$ encodes the fixed point data of the periodic transformation $t$, so by a result of J. Nielsen~\cite{N1} (or by a subsequent result of A. Edmonds~\cite[Theorem 1.3]{AE}), $t$ and $t'$ have to be conjugate by an orientation-preserving homeomorphism $\phi$. Let $\O'$ be the quotient orbifold of the action $t'$, and $\rho': \pi_1^{orb}(\O') \to C_n$ be the induced representations. Then $\phi$ induces a map $\varphi_{\#} : \pi_1^{orb}(\O) \to \pi_1^{orb}(\O')$ such that $\rho'\circ \varphi_{\#} = \rho$ as in~\cite[Theorem 2.1]{MK1}. If $\gamma$ is a loop around a cone point in $\O$, then $\rho(\gamma)$ is a loop around a cone point in $\O'$, and these cone points are associated to the same pair in $\D$ since $\rho'(\varphi_{\#}(\gamma)) = \rho(\gamma)$. Hence, $\phi$ maps $\p(t)$ to $\p(t')$ and $\o(t)$ to $\o(t')$ as in Definition \ref{defn:eq_perm_actions}.  Furthermore, $\o_t = \o_{\D} = \o_{t'}$ by construction, and hence the permuting data set determines $t$ upto equivalence.
\end{proof}

\section{Nonseparating multicurves}
\label{sec:nonsepmulcurves}

Recall that a multicurve $\C$ is said to be nonseparating if it does not contain any pseudo-nonseparating submulticurves. In this section, we establish that a root of $t_{\C}$ corresponds to a special kind of permuting action on the connected surface $S_g(\C)$. 

\begin{defn}\label{defn:orbit_equiv_sep}
Let $t_i$ be a permuting $(n_i,r_i,k_i)$-action on $S_{g_i}$ for $i =1,2$. Two orbits $\o_i \in \o(t_i)$ are said to be \textit{equivalent} (in symbols, $\o_1 \sim \o_2$) if
\begin{enumerate}[(i)]
\item  $|\o_1| = |\o_2|$, and 
\item if $|\o_1| < n := \lcm(n_1,n_2)$, then we further require that $$\theta_{\o_1}(t_1) + \theta_{\o_2}(t_2) \equiv 2\pi/n \pmod{2\pi}.$$
\end{enumerate}
\end{defn}
\noindent In this section, we will only need the case when $t_1 = t_2$, but we will need the general case in Section~\ref{sec:sepmulcurves}. 

%\begin{defn}\label{def:orbit_equiv}
%Let $t$ be a permuting $(n,r,k)$-action on $S_g$, and let $\o_1, \o_2 \in \o(t)$ be two distinguished orbits. We say that $\o_1$ is equivalent to $\o_2$ (denoted by $\o_1\sim \o_2$) if  
%\begin{enumerate}[(i)]
%\item $|\o_1| = |\o_2|$, and 
%\item if $|\o_1| < n$, we further require that $\theta_{\o_1}(t) + \theta_{\o_2}(t) \equiv 2\pi/n \pmod{2\pi}$.
%\end{enumerate}
%%Note that if $t$ is a free permuting $(n,0,k)$-action then any two of its distinguished orbits are equivalent. 
%\end{defn}
%\noindent We will see in Theorem~\ref{thm:actions-nonseproots} the reason for the distinction made between orbits of size $n$ and those of size $<n$ in the above definition.

\begin{defn}
\label{def:nesnl2m}
Let $\C$ be a nonseparating multicurve in $S_g$. A permuting $(n,2r,2k)$-action $t$ on $S_g(\C)$ is said to be \textit{nonseparating with respect to $\C$} if there exists a $(r,k)$-partition $\p_{r,k}(\C)$ of $\C$ such that
\begin{enumerate}[(i)]
\item there exists $r$ mutually disjoint pairs $\{P_i,P_i'\}$ of distinguished fixed points in $\p(t)$ such that $\theta_{P_i}(t) + \theta_{P_i'}(t) = 2\pi/n$ modulo $2\pi$, for $1 \leq i \leq r$, and
\item there exists $k$ mutually disjoint pairs $\{\o_i,\o_i'\}$ of distinguished nontrivial orbits in $\o_t$ such that $\o_i\sim \o_i'$, for $1 \leq i \leq k$,  and \item $S(\p_{r,k}) = S(t)$. 
\end{enumerate}
\end{defn}

\begin{thm}
\label{thm:actions-nonseproots}
Let $\C$ be a nonseparating multicurve in $S_g$. Then for $n \geq 1$, equivalence classes of  permuting $(n,2r,2k)$-actions on $S_g(\C)$ that are nonseparating with respect to $\C$ correspond to the conjugacy classes in $\Mod(S_g)$ of $(r,k)$-permuting roots of $t_{\C}$ of degree $n$.
\end{thm}

\begin{proof}
First, we shall prove that a conjugacy class of an $(r,k)$-permuting root $h$ of $t_{\C}$ of degree $n$ yields an equivalence class of a permuting $(n,2r,2k)$-action that is nonseparating with respect to $\C$. We assume without loss of generality that $r,k > 0$, with the implicit understanding that, when either of them is zero, the corresponding arguments may be disregarded.

Let $\C_{r,k}(h) = \{\C_1', \C_2',\ldots, \C_r', \C_1,\C_2,\ldots, \C_k\}$ be the $(r,k)$-partition of $\C$ associated with $h$. Choose a closed tubular neighborhood $N$ of $\C$, and consider $S_g(\C)$ as in Definition \ref{defn:sgc}. By isotopy, we may assume that $t_{\C}(\C)=\C$, $t_{\C}(N)=N$, and  $t_{\C}\vert_{\widehat{S_g(\C)}}=\text{id}_{\widehat{S_g(\C)}}$. Suppose that $h$ is a root of $t_{\C}$ of degree $n$, then by Remark \ref{rem:root_isotopy}, we may assume that $h$ preserves $\mathcal{C}$ and takes $N$ to $N$.

By the Nielsen-Kerckhoff theorem~\cite{K1}, $\hat{t} := h\vert_{\widehat{S_g(\C)}}$ is isotopic to a homeomorphism whose $n^{th}$ power is $\text{id}_{\widehat{S_g(\C)}}$. So we may change $h$ by isotopy so that $\hat{t}^n = \text{id}_{\widehat{S_g(\C)}}$. We fill in the $2m$ boundary circles of $\widehat{S_g(\C)}$ with disks and extend $\hat{t}$ to a homeomorphism $t$ on $S_g(\C)$ by coning. Thus $t$ defines a effective $C_n$-action on $S_g(\C)$, where $C_n=\langle t\;\vert\;t^n=1\rangle$.

The $C_n$-action $t$ fixes the centers $P_i$ and $P_i'$ of the $2r$ disks $D_i$ and $D_i'$, $1 \leq i \leq r$, of $\overline{S_g\setminus \widehat{S_g(\C)}}$ whose boundaries are the components of $\partial N$ which are fixed by $t$. The orientation of $S_g$ determines one for $S_g(\C)$, so we may speak of directed angles of rotation about the centres of these disks. Since $h^n = t_{\C}$, it follows from  \cite[Theorem 2.1]{MK1} that
$$
\theta_{P_i}(t) + \theta_{P_i'}(t) = 2\pi/n \pmod{2\pi},
$$
as illustrated in Figure~\ref{fig:twist} below. 
\begin{figure}[h]
\labellist
\tiny
\pinlabel $P_i$ at 75 80
\pinlabel $P_i'$ at 270 115
\pinlabel $\theta_{P_i}(t)$ at 85 140
\pinlabel $\theta_{P_i'}(t)$ at 295 52
%\pinlabel $q_1$ at 686 103
%\pinlabel $q_2$ at 645 75
\endlabellist
\centering
\includegraphics[width = 50 ex]{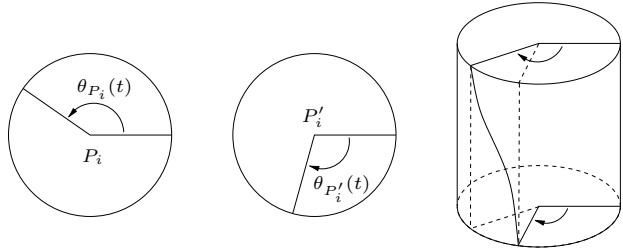}
\caption{Angle compatibility at each pair $\{P_i,P_i'\} \subset \p(t)$.}
\label{fig:twist}
\end{figure}

The remaining disks occuring in $\overline{S_g \setminus \widehat{S_g(\C)}}$ form $k$ pairs of orbits of sizes $m_1, m_2,\ldots, m_k$ where $S(\C_{r,k}(h)) = \llbrace m_1,m_2,\ldots, m_k\rrbrace$. For $1 \leq j \leq k$, we denote the centres of these pairs of disks by $Q_{i,j}$ and $Q_{i,j}'$, and the orbits of these centres by $\o_j$ and $\o_j'$. Thus $t$ is a permuting $(n,2r,2k)$-action with $\p(t) = \{P_1,P_1',\ldots,P_r,P_r'\}$ and $\o(t) = \{\o_1,\o_1',\ldots,\o_k,\o_k'\}$.

It remains to show that $\o_i \sim \o_i'$ for each $i$. By construction, $|\o_i|=|\o_i'| = m_i$, and if $m_i =n$, then $\o_i\sim \o_i'$ holds trivially. If not, then we need to show that the angle compatibility condition from Definition \ref{def:orbit_equiv} holds. Write $\o_i= \{Q_{i,1},Q_{i,2},\ldots, Q_{i,m_i}\}$, $\o_i' = \{Q_{i,1}',Q_{i,2}',\ldots, Q_{i,m_i}'\}$ and note that $h^{m_i}$ is an $(n/m_i)^{\text{th}}$ root of $t_{\C}$ such that $h^{m_i}(c_{i,1}) = c_{i,1}$. Hence,
$$
\theta_{Q_{i,1}}(t^{m_i}) + \theta_{Q_{i,1}'}(t^{m_i}) = 2\pi/(n/m_i) \pmod{2\pi},
$$
which implies that
$$
m_i\theta_{\o_i}(t) + m_i\theta_{\o_i'}(t) \equiv 2\pi/(n/m_i) \pmod{2\pi}.
$$
Since $t^n$ is a restriction of a single Dehn twist, it follows that 
$$
\theta_{\o_i}(t) + \theta_{\o_i'}(t) \equiv 2\pi/n \pmod{2\pi}.
$$
Hence $\o_i \sim \o_i'$, and we obtain a permuting $(n,2r,2k)$-action $t$ on $S_g(\C)$ that is nonseparating with respect to $\C$. 

Now suppose $h_1, h_2\in \Mod(S_g)$ are two roots of $t_{\C}$ that are conjugate in $\Mod(S_g)$ via $\Phi \in \Mod(S_g)$, and let $t_s$ denote the finite order homeomorphisms on $S_g(\C)$ induced by $h_s$, for $s = 1,2$. Then $t_{\C} = \phi t_{\C}\Phi^{-1} = t_{\Phi(\C)}$, so we may assume upto isotopy that $\Phi(\C) = \C$ (as in Remark \ref{rem:root_isotopy}) and that $\Phi(N) = N$. We extend $\Phi\mid_{\widehat{S_g(\C)}}$ to an element $\phi \in \Mod(S_g(\C))$ by coning. Now, $\phi$ maps $\p(t_1)$ to $\p(t_2)$, and $\o(t_1)$ to $\o(t_2)$ bijectively as in Definition~\ref{defn:eq_perm_actions}. Since the $h_s$ and $\Phi$ all preserve $N$, $\phi t_1\phi^{-1}$ is isotopic to $t_2$ preserving $\p(t_2)$ and $\o(t_2)$. Furthermore, for each $\o \in \o(t_1)$, $p(\phi(\o)) = p(\o)$. Hence, $\o_{t_1} = \o_{t_2}$ and so $t_1$ and $t_2$ will be equivalent permuting $(n,2r,2k)$-actions.

Conversely, given a permuting $(n,2r,2k)$-action $t$ on $S_g(\C)$ that is nonseparating with respect to $\C$, we can reverse the argument to produce the $(r,k)$-permuting root $h$. Let $\p(t) = \{P_1, P_1', \ldots, P_r, P_r'\}$ and $\o(t) = \{\o_1,\o_1', \ldots, \o_k,\o_k'\}$. For $1\leq i\leq r$, we remove disks $D_i$ and $D_i'$ invariant under the action of $t$ around the $P_i$ and $P_i'$ and attaching $r$ annuli to obtain the surface $S_g(\C\setminus \C')$.  The condition on the angles $\{\theta_{P_i}(t), \theta_{P_i'}(t)\}$ ensures that the rotation angles work correctly to allow an extension of $t$ to obtain an $h_0$ with $h_0^n = t_{\C'}$ in $\Mod(S_g(\C\setminus \C'))$, where $\C'= \cup_{i=1}^r \C_i'$.

Now write $\o_1 = \{Q_{1,1}, Q_{1,2}, \ldots, Q_{1,m_1}\}$ and consider disks $D_{1,i}$ around $Q_{1,i}$ such that $t(D_{1,i}) = t(D_{1,i+1})$. Similarly, write $\o_1'=\{Q_{1,1}',Q_{1,2}',\ldots, Q_{1,m_1}'\}$ and consider disks $D_{1,i}'$ as earlier. Then we attach $m_1$ annuli connecting $\partial D_{1,i}$ to $\partial D_{1,i}'$. Each such annulus contains a nonseparating curve $c_{1,i}$, which is unique unto isotopy. Repeating this process for $1\leq i\leq m_1$, we obtain the surface $S_g(\C\setminus (\C'\cup\C_1))$. Since $t(D_{1,i}) = D_{1,i+1}$, we may extend the homeomorphism $h_0$ to a homeomorphism $\widetilde{h_0} \in \Mod(S_g(\C'\cup \C_1))$, which cyclically permutes the $c_{1,i}$. If $|\o_1| = |\o_1'| = n$, then define $h_1 := \widetilde{h_0}t_{c_{1,1}}$. Otherwise, since $\o_1 \sim \o_1'$, the difference in the turning angles around $Q_{1,i}$ and $Q_{1,i}'$ is $2\pi/n$. Let $\widetilde{h_1}$ be the $(1/n)^{th}$-twist around $c_{1,1}$. Now $h_1 := \widetilde{h_0}\widetilde{h_1}$ is an $(r,1)$-permuting root of $t_{\C' \cup \C_1}$ of degree $n$ in $\Mod(S_g(\C \setminus (\C' \cup \C_1)))$. We now repeat this process inductively to obtain an $(r,k)$-permuting root $h:= h_k\in \Mod(S_g)$ of $t_{\C}$ of degree $n$.

It remains to show that the resulting root $h$ of $t_{\C}$ is determined up to conjugacy. Suppose $t_1$ and $t_2$ are two equivalent $(n,2r,2k)$-actions on $S_g(\C)$ that are nonseparating with respect to $\C$ with $\p(t_s) = \{P_{s,1},P_{s,2},\ldots, P_{s,r}\} \text{ and }\o(t_s) = \{\o_{s,1},\o_{s,2},\ldots, \o_{s,k}\}$, for $s=1,2$. Let $\phi \in \Mod(S_g(\C))$ be an orientation-preserving homeomorphism satisfying the conditions in Definition \ref{defn:eq_perm_actions}. Then repeating the argument from \cite[Theorem 2.1]{MK1}, $\phi$ extends to a homeomorphism $\Phi_0 \in \Mod(S_g(\C\setminus \C'))$ such that $\Phi_0 h_{1,0} \Phi_0^{-1} = h_{2,0}$, where $h_{s,0}$ is the root of $t_{\C'}$ obtained from $t_s$, for $s=1,2$, as above. Furthermore, since $\phi$ maps $\o_{1,i}$ to $\o_{2,i}$ as in Definition \ref{defn:eq_perm_actions}, we may once again extend $\Phi_0$ to a homeomorphism $\Phi \in \Mod(S_g)$ satisfying $\Phi h_1 \Phi^{-1} = h_2$, where $h_s$ is the root of $t_{\C}$ obtained from $t_s$, for $s=1,2$.
\end{proof}

\begin{defn}
\label{def:nonsep-dataset}
A permuting $(n,2r,2k)$-data set $(\D,\o_{\D})$ is called \emph{nonseparating} if
\begin{enumerate}[(i)]
\item 
$\displaystyle \D = (n,g_0,(a_1,a_1',\ldots,a_r,a_r');(c_1,n_1),\ldots,(c_s,n_s)),$ 
where  \\ $a_i + a_i' \equiv a_ia_i' \bmod n,$ for $1 \leq i \leq r$
\item For each $(p,m_p) \in \o_{\D}$, there exists $(p',m_p') \in \o_{\D}$ such that
$m_p = m_p'$ and 
$$
\theta(p) + \theta(p') = \begin{cases}
0 &: p = p' = (0,1) \\
2\pi/n \pmod{2\pi} &: \text{otherwise.}
\end{cases}
$$

\end{enumerate}
\end{defn} 

\noindent We now deduce the following theorem from Theorems \ref{thm:action_triple_correspondence} and \ref{thm:actions-nonseproots}.

\begin{thm}
Let $\C$ be a nonseparating multicurve in $S_g$. For $n\geq 1$, equivalence classes of nonseparating permuting $(n,2r,2k)$-data sets of genus $g_{\C}$ correspond to equivalence classes of permuting $(n,2r,2k)$-actions on $S_g(\C)$ that are nonseparating with respect to $\C$.
\end{thm}

\begin{cor}\label{cor:nonsepdsets-nonseproots}
Let $\C$ be a nonseparating multicurve in $S_g$. For $n\geq 1$, equivalence classes of nonseparating permuting $(n,2r,2k)$-data sets of genus $g_{\C}$ correspond to conjugacy classes in $\Mod(S_g)$ of $(r,k)$-permuting roots of $t_{\C}$ of degree $n$.
\end{cor}

If $\C$ is a nonseparating multicurve of size $m$, then $g_{\C} = g-m$. Hence we obtain the following restrictions on the degree of a root of $t_{\C}$. 

\begin{cor}
\label{cor:bound_nonsep}
Let $\C$ be a nonseparating multicurve in $S_g$ of size $m$, and let $h$ be an $(r,k)$-permuting root of $t_{\C}$ of degree $n$. 
\begin{enumerate}[(i)]
\item If $r\geq 0$, then
$$
n\leq \begin{cases}
4(g-m) + 2, &\text{if } g-m \geq 1 \\
g, &\text{if } g = m.
\end{cases}
$$
Furthermore, if $g=m$, then this upper bound is realizable.
\item If $r\geq 1$, then $n$ is odd.
\item If $r=1$, then $n\leq 2(g-m)+1$. %Furthermore, if $g_{\C}\geq 2$, then $n \leq 2g_{\C}$
\item If $r\geq 2$, then $\displaystyle n \leq \frac{g-m+r-1}{r-1}$.
\end{enumerate}
\end{cor}
\begin{proof}
If $r\geq 0$ and $g-m\geq 1$, then by a result of Wiman \cite[Theorem 6]{H1}, the highest order of a cyclic action on $S_g(\C)$ is $4(g-m) + 2$.  If $r\geq 0$ and $g=m$, then consider the permuting $(n,2r,2k)$-action $t$ on $S_0$ that is nonseparating with respect to $\C$ guaranteed by Theorem \ref{thm:actions-nonseproots}. Since $t \in \Mod(S_0)$ is an element of order $n$, it must be a rotation by $2\pi/n$ radians. Since the two fixed points of this action are not compatible in the sense of Definition \ref{def:nesnl2m}, $r=0$ and every non-trivial orbit has size $n$. Hence, $m = nk$ and so $n\mid m$, and in particular, $n\leq m=g$. Furthermore, if $m = g$, let $t$ be the rotatiuon of the sphere by $2\pi/n$, and $c \in \C$ be any curve. Then $h := tt_c$ is a root of $t_{\C}$ of degree $m$ and (i) follows. 

(ii) and (iii) follow from \cite[Corollary 2.2]{MK1}. If $r\geq 2$, then consider the corresponding permuting $(n,2r,2k)$-data set from Theorem \ref{thm:action_triple_correspondence}. Note that the genus $(g-m)$ of the permuting data set is given by
$$\frac{2-2(g-m)}{n} = 2-2g_0 + 2r\left(\frac{1}{n}-1 \right) + \sum_{j=1}^s \left(\frac{1}{n_j} - 1 \right).$$ Since $(1-1/n_j) \geq 0$ and $g_0 \geq 0,$ we have $(g-m) - 1 +r \geq n(-1+r),$ from which (iv) follows.
\end{proof}

\section{Separating Multicurves}
\label{sec:sepmulcurves}

A separating multicurve $\C$ in $S_g$ is the disjoint union of finitely many pseudo-nonseparating curves. In this case $S_g(\C)$ is disconnected, so we will require multiple finite order actions on the individual components of $S_g(\C)$ to come together to form a root of $t_{\C}$ on $S_g$. To improve the exposition, we begin with the case $\C = \C^{(m)}$ so that $S_g(\C)$ has two components. 
%
%\begin{defn}\label{defn:orbit_equiv_sep}
%Let $t_i$ be a permuting $(n_i,r_i,k_i)$-action on $S_{g_i}$ for $i =1,2$. Two orbits $\o_i \in \o(t_i)$ are said to be \textit{equivalent} (in symbols, $\o_1 \sim \o_2$) if
%\begin{enumerate}[(i)]
%\item  $|\o_1| = |\o_2|$, and 
%\item if $|\o_1| < n := \lcm(n_1,n_2)$, then we further require that $$\theta_{\o_1}(t_1) + \theta_{\o_2}(t_2) \equiv 2\pi/n \pmod{2\pi}.$$
%\end{enumerate}
%\end{defn}

\begin{defn}\label{defn:mpair_permact}
Equivalence classes $\l t_i\r$ of permuting $(n_i,r_i,k_i)$-actions on $S_{g_i}$, for $i = 1,2$, are said to form an \textit{(r,k)-compatible pair $(\l t_1 \r, \l t_2 \r)$ of degree $n$} for integers $r,k \geq 0$, if 
\begin{enumerate}[(i)]
\item $n = \lcm(n_1,n_2)$, and
\item there exists $\{P_{i,1}, P_{i,2},\ldots, P_{i,r}\} \subset \p(t_i)$ such that for $1\leq j\leq r$,
$$
\displaystyle \theta_{P_{1,j}}(t_1) + \theta_{P_{2,j}}(t_2) = \frac{2\pi}{n} \pmod{2\pi} \text{, and }
$$
\item there exists $\{\o_{i,1},\o_{i,2},\ldots, \o_{i,k}\} \subset \o(t_i)$ for $i=1,2$ such that $\o_{1,j} \sim \o_{2,j}$ as in Definition~\ref{defn:orbit_equiv_sep}, for $1\leq j\leq k$.
\end{enumerate}
If
$$
\alpha := \sum_{j=1}^k |\o_{1,j}| = \sum_{j=1}^k |\o_{2,j}|
$$
then the number $g := g_1+g_2+r+\alpha - 1$ is called the \emph{genus} of the pair $(\l t_1\r, \l t_2\r)$. Note that two actions are $(1,0)$-compatible if they are compatible as nestled actions in the sense of \cite[Definition 3.2]{KR1}.
\end{defn}

\begin{notation}
Let $(\l t_1\r, \l t_2\r)$ be an $(r,k)$-compatible pair of degree $n$ as in Definition \ref{defn:mpair_permact}. We write $\p(t_1,t_2):= \{P_{1,1},P_{1,2},\ldots, P_{1,r}\}$ and $\o(t_1,t_2):= \{\o_{1,1},\o_{1,2},\ldots, \o_{1,k}\}$. We define $\p(t_2,t_1)$ and $\o(t_2,t_1)$ similarly.
\end{notation}

\begin{lemma}\label{lem:powers_of_twists}
Let $\C=\{c_1,c_2,\ldots, c_m\}$ be a multicurve on $S_g$, and $N_i$ be annular neighbourhood of $c_i$. Write $N = \sqcup_{i=1}^m N_i$, and suppose $t\in \Mod(S_g)$ is such that $t\lvert_{S_g\setminus N} = \text{id}_{S_g\setminus N}$, then $\exists d_1,d_2,\ldots, d_m\in \mathbb{N}\cup\{0\}$ such that $t = t_{c_1}^{d_1}\ldots t_{c_m}^{d_m}$.
\end{lemma}
\begin{proof}
Since $t\lvert_{S_g\setminus N} = \text{id}_{S_g\setminus N}$, it follows that $t$ fixes $\partial N$. The lemma now follows from the fact \cite[Lemma 4.1.A, Chapter 12]{DS} that $$\Mod(N\text{fix}(\partial N)) \cong \oplus_{i=1}^m \Mod(N_i\text{fix}(\partial N_i)) = \oplus_{i=1}^m \langle t_{c_i}\rangle.$$
\end{proof}

\begin{thm}\label{thm:homologous_root_action}
Suppose that $S_g = S_{g_1}\#_{\C} S_{g_2}$, where $\C = {\C}^{(m)}$. Then $(r,k)$ -permuting roots of $t_{\C}$ of degree $n$ correspond to the $(r,k)$-compatible pairs $(\l t_1\r, \l t_2\r)$ of equivalence classes of permuting $(n_i,r_i,k_i)$-actions on the $S_{g_i}$, of degree $n$.
\end{thm}

\begin{proof}
As before, we assume $m>1$ and first show that every $(r,k)$-permuting root $h$ of $t_{\C}$ of degree $n$ yields a compatible pair $(\l t_1 \r, \l t_2 \r)$ of degree $n$. Consider the $(r,k)$-partition $\C_{r,k}(h) = \{\C_1',\ldots \C_r',\C_1,\ldots, \C_k\}$ of $\C$ induced by $h$. Let $\widehat{S_{g_s}}$ for $s=1,2$ denote the two components of $\widehat{S_g(\C)}$. Let $N$ be a closed annular neighborhood of $\C$. By isotopy, we may assume that 
$t_{\C}(\C) = \C$, $t_{\C}(N) = N$ , and $t_{\C}\vert_{\widehat{S_{g_s}}} = id_{\widehat{S_{g_s}}}$. 
Putting $\widehat{t_s} = h\vert_{\widehat{S_{g_s}}}$, we may assume up to isotopy that 
$\widehat{t_s}^n\vert_{\widehat{S_{g_s}}} = \text{id}_{\widehat{S_{g_s}}}$ for $s=1,2$. 

Let $n_s$ be the smallest positive integer such that $\widehat{t_s}^{n_s}=\text{id}_{\widehat{S_{g_s}}}$ for $s=1,2$, and let $q=\lcm(n_1,n_2)$. Then $t:= h^q$ satisfies the hypotheses of Lemma \ref{lem:powers_of_twists}. Hence, $\exists d_c \in \mathbb{N}\cup\{0\}$ such that
$$
h^q = \prod_{c\in \C} t_c^{d_c}
$$
Since $h^n\lvert_{\widehat{S_{g_1}}} = \text{id}_{\widehat{S_{g_1}}}$ it follows that $n_1\mid n$, and similarly $n_2\mid n$. Hence, $q\mid n$ and so
$$
\prod_{c\in \C}t_c = t_{\C} = (h^q)^{n/q} = \prod_{c\in \C}t_c^{nd_c/q}
$$
Fix $c\in \C$ and restrict the functions on both sides of this equation to a closed annular neighbourhood of $c$ disjoint from other curves in $\C$. As in Remark \ref{rem:root_isotopy}, we see that $nd_c/q = 1$ and hence $n = q = \lcm(n_1,n_2)$. 

We fill in $\partial \widehat{S_{g_s}}$ with disks to obtain the closed oriented surfaces $S_{g_s}$ for $s=1,2$. We then extend $\widehat{t_s}$ to the $S_{g_s}$ by coning. Thus $t_s$ defines an effective $C_{n_s}$-action on the $S_{g_s}$, where $n_s \mid n$ for $s =1,2$, and $n = \lcm(n_1,n_2)$. 

When $r > 0$, the homeomorphism $t_s$ fixes the center points $\{P_{s,1},P_{s,2},\ldots, P_{s,r}\}$ of $r$ disks in $\overline{S_{g_s}\setminus \widehat{S_{g_s}}}$ for $s =1,2$. Hence we may write $\p(t_1,t_2) = \{P_{1,1},P_{1,2},\ldots, P_{1,r}\}$ and $\p(t_2,t_1) = \{P_{2,1},P_{2,2},\ldots, P_{2,r}\}$. For $1\leq j\leq r$, the proof of \cite[Theorem 3.4]{KR1} implies that the corresponding turning angles around $P_{1,j}$ and $P_{2,j}$ must be compatible in the sense of condition (ii) of Definition \ref{defn:mpair_permact}.

When $k > 0$, let $\C_i = \{c_{i,1},c_{1,2},\ldots, c_{i,m_i}\}$, for $1 \leq i \leq k$. Associated with each curve $c_{i,j} \in \C_i$, is a disk 
$D_{i,j}^s$ in each $\overline{S_{g_s}\setminus \widehat{S_{g_s}}}$ for $s = 1,2$. The centers $Q_{i,j}^s$ of the $m_i$ disks $D_{i,j}^s$ for $1 \leq j \leq m_i$ form an orbit $\o_{s,i}$ in $S_{g_s}$ for $s = 1,2$.  Thus we obtain a collection $\o_{t_s} = \{\o_{s,1}, \ldots , \o_{s,k}\}$ of $k$ distinguished nontrivial orbits on $S_{g_s}$ for $s = i,j$. It remains to show that $\o_{1,i} \sim \o_{2,i}$, for $1\leq i\leq k$, but the argument for this is similar to that of Theorem~\ref{thm:actions-nonseproots}. Hence, the pair $(\l t_1\r, \l t_2\r)$ forms an $(r,k)$-compatible pair of degree $n$.

The argument for the converse, and the fact that the resulting root is determined upto conjugacy, is analogous to that of Theorem~\ref{thm:actions-nonseproots}.
\end{proof}

\begin{defn}
\label{defn:doubpermdspair}
Permuting data sets $\widetilde{\D_i} := (\D_i,\o_{\D_i})$ for $i=1,2$, where 
$$\D_i = (n_i, \widetilde{g_i}, (a_{i,1},\ldots, a_{i,r_i}); ((c_{i,1},y_{i,1}) ,\ldots, (c_{i,s_i},y_{i,s_i})),$$ are said to be \textit{$(r,k)$-compatible of degree $n$} for integers $r \geq 0$ and $k \geq 0$, if 
\begin{enumerate}[(i)]
\item $n = \lcm(n_1,n_2)$,
\item $r  \leq r_i$, for $i=1,2$, and for $1 \leq j \leq r$,
$$a_{1,j}+a_{2,j} \equiv a_{1,j}a_{2,j} \pmod n,$$ 
\item there exists $S_i \subseteq \o_{\D_i}$ such that
\begin{enumerate}
\item 
$$
k = \displaystyle \sum_{(p,m_p) \in S_1} m_p = \sum_{(q,m_q) \in S_2} m_q,
$$
\item for every $(p,m_p) \in S_1$, there exists $(q,m_q) \in S_2$ such that
$$
\theta(p)+\theta(q) \equiv 
\begin{cases}
0 &: \text{ if } p = q = (0,1) \\
2\pi/n \pmod{2\pi} &: \text{ otherwise}
\end{cases}
$$

\end{enumerate}
\end{enumerate}
If $$\alpha := \sum_{((a,b),m_p)\in S_1} n_1/b = \sum_{((c,d),m_q) \in S_2} n_2/d$$ and
if $g_i$ is the genus of $\D_i$, then the number $g = g_1+g_2+r+\alpha-1$ is called \textit{genus} of the pair 
$(\widetilde{\D_1},\widetilde{\D_2})$.
\end{defn}

\noindent The next theorem and its corollary follow directly from Theorems~\ref{thm:action_triple_correspondence} and \ref{thm:homologous_root_action}.

\begin{thm}
\label{thm:mcomp_per_pairs-mcompdspairs}
Suppose that $S_g = S_{g_1}\#_{\C} S_{g_2}$, where $\C = {\C}^{(m)}$. Then $(r,k)$-compatible pairs $(\l t_1\r, \l t_2\r)$ of equivalence classes of permuting $(n_i,r_i,k_i)$-actions on the $S_{g_i}$ correspond to $(r,k)$-compatible pairs $(\widetilde{\D_1}, \widetilde{\D_2})$ of data sets of genus $g$, where $\widetilde{D_i}$ is a permuting $(n_i,r_i,k_i)$-data set of genus $g_i$.
\end{thm}

\begin{cor}
Suppose that $S_g = S_{g_1}\#_{\C} S_{g_2}$, where $\C = {\C}^{(m)}$. Then conjugacy classes of $(r,k)$-permuting roots of $t_{\C}$ of degree $n$ correspond to $(r,k)$-compatible pairs $(\widetilde{\D_1},\widetilde{\D_2})$ of degree $n$ and genus $g$, where $\widetilde{D_i}$ is a permuting $(n_i,r_i,k_i)$ data set of genus $g_i$.
\end{cor}

We now consider the case of a $(0,k')$-permuting root $h$ of $t_{\C}$ where $\C = \C^{(k)}(m)$. Here, the restriction of $h$ to $S_g(\C)$ induces a non-trivial permutation of the components of $S_g(\C)$. Since $h$ is a homeomorphism, it maps one component to another of the same genus, thus inducing an action on each subsurface of the form $\S_g(m) \subset S_g(\C)$. Therefore, we generalize the notion of a permuting $(n,r,k)$-action to encompass the action on $\S_g(m)$ induced by $h$.

\begin{defn} \label{defn:permuting_action_sgm}
Fix integers $g \geq 0$ and $m \geq 1$.
\begin{enumerate}[(i)]
\item A homeomorphism $\sigma_m : \S_g(m) \to \S_g(m)$ is said to be \textit{essential} if for all $i$, $\sigma_m\vert_{S_g^i} (S_g^i)= S_g^{i+1}$. 
\noindent An essential homeomorphism $\sigma_m$ on $\S_g(m)$ can be viewed simply as the $m$-cycle $(1,2,\ldots,m)$ permuting its $m$ components $S_g^i$. 
\item An orientation-preserving $C_n$-action $t$ on $\S_g(m)$ is said to be a \emph{permuting $(n,r,k)$-action} if $m \mid n$ and there exists an essential homeomorphism $\sigma_m : \S_g(m) \to \S_g(m)$ such that $t = \sigma_m \circ \widetilde{t}$, where $\widetilde{t}$ is a permuting $(n,r,k)$-action on each $S_g^i$. 
\end{enumerate}
\end{defn}

\begin{rem}\label{rem: permuting_action_sgm}
Suppose that $\widetilde{t} \in \Mod(S_g)$ defined a permuting $(n,r,k)$-action and $t = \sigma_m\circ \widetilde{t}$, then $t_i := t^m\lvert_{S_g^i} \in \Mod(S_g^i)$ defines a permuting $(n/m,\widetilde{r},\widetilde{k})$-action on $S_g^i$. Furthermore, all the $t_i$ are conjugate to each other via $\sigma_m$.

Conversely, if $t'\in \Mod(S_g)$ is a permuting $(n/m, \widetilde{r},\widetilde{k})$-action on $S_g$ which has an $m^{th}$ root $\widetilde{t} \in \Mod(S_g)$, then the map $t := \sigma_m\circ\widetilde{t}$ defines a permuting $(n,r,k)$-action on $\S_g(m)$. Thus, a permuting $(n,r,k)$-action on $\S_g(m)$ corresponds to a permuting $(n/m,\widetilde{r},\widetilde{k})$-action on $S_g$ which has an $m^{th}$ root in $\Mod(S_g)$.
\end{rem}

\begin{defn}
Let $t_1$ and $t_2$ be two permuting $(n,r,k)$-actions on $\S_g(m)$. 
Then we say $t_1$ is \emph{equivalent} to $t_2$ if ${t_1}^m\vert_{S_g^i}$ and ${t_2}^m\vert_{S_g^i}$ are equivalent as permuting $(n/m,\widetilde{r},\widetilde{k})$-actions on $S_g^i$ in the sense of Definition~\ref{defn:eq_perm_actions}. 
\end{defn}

%\begin{defn}
%Let $t_1$ be a permuting $(n_1,r_1,k_1$)-action on $S_{g_1}$ and $t_2$ be a permuting $(n_2,r_2,k_2)$-action on $\S_{g_2}(m)$ such that $t_2 = \sigma_m \circ \widetilde{t_2}$, where $\widetilde{t_2}$ is a permuting $(n_2,r_2,k_2)$-action on $S_{g_2}$. Let $\o_1 \in \o(t_1)$ and $\o_2 \in \o(\tilde{t_2})$, then we say that $\o_1$ and $\o_2$ are equivalent (denoted by $\o_1\sim \o_2$) if
%\begin{enumerate}[(i)]
%\item $|\o_1| = |\o_2|$, and
%\item if $n_1 \neq n_2$, then we further require that
%$$
%\theta_{\o_1}(t_1) + \theta_{\o_2}(\tilde{t_2}) = 2\pi/n \pmod{2\pi},
%$$
%where $n = \lcm(n_1,n_2)$.
%\end{enumerate}
%\end{defn}

\begin{defn}
\label{defn:uv_comp_pair}
Let $u,v\geq 0$ be fixed integers, $t_1$ be a permuting $(n_1,r_1,k_1)$-action on $S_{g_1}$ and $t_2$ be a permuting $(n_2,r_2,k_2)$-action on $\S_{g_2}(m)$ such that $t_2 = \sigma_m \circ \widetilde{t_2}$ as in Remark \ref{rem: permuting_action_sgm}. Then the equivalence classes $(\l t_1\r, \l t_2\r)$ are said to form a \textit{(u,v)-compatible pair} if 
\begin{enumerate}[(i)]
\item for $1 \leq i \leq m$, $(\l {t_1}^m \r, \l t_2^m\vert_{S_{g_2}^i} \r)$ is an $(u,v)$-compatible pair of degree $n/m$, where $n=\lcm(n_1,n_2)$, 
\item for $1\leq i\leq m, 1\leq j \leq u,$ there mutually disjoint pairs $\{P_{i,j}^1,P_{i,j}^2\}$, where $P_{i,j}^1 \in \p(t_1^m)$ and $P_{i,j}^2 \in \p(t_2^m\vert_{S_{g_2}^i})$ such that 
$$
\theta_{P_{i,j}^1}(t_1) + \theta_{P_{i,j}^2}(\widetilde{t_2}) = 2\pi/n \pmod{2\pi}, \text{ and } 
$$
\item for $1\leq i\leq m, 1\leq j \leq v,$ mutually disjoint pairs $\{\o_{i,j}^1, \o_{i,j}^2\}$, where $\o_{i,j}^1 \in \o(t_1)$ and $\o_{i,j}^2 \in \o(\widetilde{t_2})$, such that
$\o_{i,j}^1 \sim \o_{i,j}^2$, as in Definition~\ref{defn:orbit_equiv_sep}.
%\item for $1\leq i\leq m$, the sets $$\p(t_1^m,t_2^m\lvert_{S_{g_2}^i}) \sqcup \left(\sqcup_{\o \in \o(t_1^m, t_2^m\lvert_{S_{g_2}^i})} \o \right)$$ are mutually disjoint.
\end{enumerate} 
The number $n = \lcm(n_1,n_2)$ is called the \textit{degree} of the pair, and the number $g = g_1 + m(g_2+k-1)$ is called the \textit{genus} of the pair. 
\end{defn}

\begin{notation}
Let $(\l t_1\r, \l t_2\r)$ be a $(u,v)$-compatible pair of degree $n$ as in Definition \ref{defn:uv_comp_pair}. We write 
\begin{enumerate}[(i)]
\item $\p(t_1, t_2) = \{P_{i,j}^1 : 1\leq i\leq m, 1\leq j\leq u\}$ and $\p(t_2,t_1) = \{P_{i,j}^2 : 1\leq i\leq m, 1\leq j\leq u\}$.
\item Similarly, we define $\o(t_1,t_2) = \{\o_{i,j}^1 : 1\leq i\leq m, 1\leq j\leq v\}$ and $\o(t_1,t_2) = \{\o_{i,j}^2 : 1\leq i\leq m, 1\leq j\leq v\}$.
\end{enumerate}
\end{notation}

\begin{lemma}\label{lem:single_orbit_actions}
Suppose that $\C = \C^{(k)}(m)$ be a multicurve in $S_g$. Then conjugacy classes of $(0,k')$-permuting roots of $t_{\C}$ of degree $n$ correspond to $(u,v)$-compatible pairs of degree $n$ and genus $g$, where $k'=u+v$.
\end{lemma}
\begin{proof}
Let $h\in\Mod(S_g)$ be a $(0,k')$-permuting root of degree $n$ on $S_g$. Then as in Theorem \ref{thm:homologous_root_action}, we obtain an effective $C_{n_1}$-action $t_1$ on $S_{g_1}$ and an effective $C_{n_2}$-action on $S_{g_2}(m)$, where $n=\lcm(n_1,n_2)$. Furthermore, $h$ restricts to an essential homeomorphism $\sigma_m : \S_g(m) \to \S_g(m)$ such that $\sigma_m\circ t_2 = t_2\circ \sigma_m$. Hence the maps
$$
t_{2,i}:= \sigma_m^{-1}t\lvert_{S_{g_2}^i} : S_{g_2}^i \to S_{g_2}^i
$$
are conjugate to each other when considered as elements of $\Mod(S_g)$, and so $t_2$ is a permuting $(n_2,r_2,k_2)$-action on $S_{g_2}(m)$ as in Definition \ref{defn:permuting_action_sgm}. Since $h^m$ is a root of $t_{\C}$ of degree $(n/m)$ that preserves each $\C^{(k)}_i$, we may restrict $h^m$ and cone to obtain maps
$$
h_i \in \Mod(\Sigma_i) \text{, where } \Sigma_i := S_{g_1}\#_{\C_i^{(k)}}S_{g_2}^i,
$$
which are pairwise conjugate to each other via $h$. Thus, it follows from Theorem \ref{thm:homologous_root_action} that there exist integers $u,v\geq 0$ such that $(\l t_1^m\r, \l t_2^m\lvert_{S_{g_2}^i}\r)$ forms a $(u,v)$-compatible pair of degree $n/m$, for $1\leq i\leq m$. 

By condition (ii) of Definition \ref{defn:mpair_permact}, there exists mutually disjoint pairs $\{P_{i,j}^1, P_{i,j}^2\}$, where $P_{i,j}^1 \in \p(t_1^m)$ and $P_{i,j}^2 \in \p(t_2^m\vert_{S_{g_2}^i})$, such that 
$$
\theta_{P_{i,j}^1}(t_1^m) + \theta_{P_{i,j}^2}(\widetilde{t_2}^m) = \frac{2\pi}{(n/m)} \pmod{2\pi}.
$$
Once again, since $h^m$ is a root of $t_{\C}$ it follows that
$$
\theta_{P_{i,j}^1}(t_1) + \theta_{P_{i,j}^2}(\widetilde{t_2}) = 2\pi/n \pmod{2\pi}.
$$
Similarly, one obtains condition (iii) of Definition \ref{defn:uv_comp_pair} as well. Note that every fixed point of $h^m$, and every orbit of $h^m$, induces one orbit of $h$, and thus $k'=u+v$. The converse is a just a matter of reversing this argument. 
\end{proof}

We now consider the case of an $(r,k)$-permuting root of $t_{\C}$ where $r,k>0$. We begin by writing $S_g$ as a connected sum of subsurfaces $S_{g_i}$ across those pseudo-nonseparating submulticurves which are preserved by $h$. The restriction of $h$ to each $S_{g_i}$ is then a $(0,k_i)$-permuting root of the Dehn twist about the  submulticurve $\C\cap S_{g_i}$. This allows us to apply Theorem \ref{lem:single_orbit_actions} to obtain finite order actions on $S_{g_i}$ such that pairs of actions on adjacent subsurfaces are compatible (in the sense of Theorem \ref{thm:homologous_root_action}).

\begin{notation}\label{not:rk_permuting_root} 
Let $\C$ be a separating multicurve in $S_g$, and let $h$ be an $(r,k)$-permuting root of $t_{\C}$. 
\begin{enumerate}[(i)]
\item We shall denote the set of all pseudo-nonseparating submulticurves of $\C$ that are preserved by $h$ by $\Fix_h(\C) = \{\E_1,\ldots,\E_{m(h)}\}$. %We shall write $\E_{r_i',k_i'}(h):= \C_{r,k}(h) \cap \E_i$.
\item Writing
$$
S_g = \csum_{i=1}^{m(h)}  (S_{g_i} \#_{\E_i} S_{g_{i+1}})
$$
and $D_i := \C \cap S_{g_i}$, we have that for each $i$, $S_{g_i} \cap \C_{r,k}(h)$ is a $(0,k_i)$-partition of $D_i$, which has the form
$$
S_{g_i} \cap \C_{r,k}(h) = \{ \C^{(k_{i,j})}_{i,j}(m_{i,j})\,:\, 1 \leq j \leq k_i\}.$$
\item For $1 \leq i \leq r+1$, we write
$$
S_{g_i} = \csum_{j=1}^{k_i} (S_{g_{i,1}} \#_{D_{i,j}} \S_{g_{i,2,j}}(m_{i,j})),
$$
where $D_{i,j} = \C_{i,j}^{(k_{i,j})}(m_{i,j})$ and $g_i = g_{i,1} + \sum_{j=1}^{k_i} m_{i,j} (g_{i,2,j} + k_{i,j} - 1)$. 
\item We will denote $\displaystyle \min_{1 \leq i \leq m(h)} g_i+g_{i+1}$ by $g(\C)$.
\end{enumerate}
\end{notation}
\noindent In Figure~\ref{fig:split-surface} below, we have the surface $S_{24}$ with the multicurve $$\C = \C^{(2)}(2) \sqcup \C^{(1)}(2) \sqcup \C^{(3)} \sqcup \C^{(1)} \sqcup \C^{(1)}(3).$$ According to  Notation~\ref{not:rk_permuting_root}, 
$$S_{24} = \left(\S_3(2) \#_{\C^{(2)}(2)} S_2 \#_{\C^{(1)}(2)} \S_1(2)\right)  \#_{\C^{(3)}} S_2 \#_{\C^{(1)}} \left(S_1 \#_{\C^{(1)}(3)}\S_3(3)\right).$$
\begin{figure}[h]
\labellist
\tiny
\pinlabel $\C^{(2)}(2)$ at 70 95
\pinlabel $\C^{(1)}(2)$ at 225 145
\pinlabel $\C^{(3)}$ at 333 40
\pinlabel $\C^{(1)}$ at 545 145
\pinlabel $\C^{(1)}(3)$ at 620 40
\endlabellist
\centering
\includegraphics[width = 70 ex]{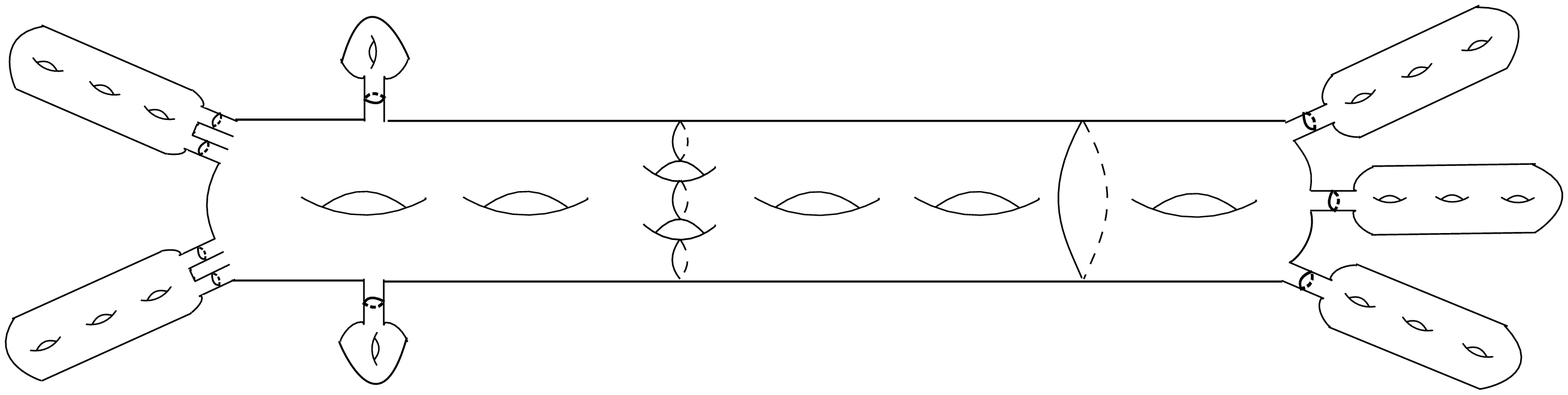}
\caption{$S_{24}$ with a separating multicurve $\C$.}
\label{fig:split-surface}
\end{figure}

\begin{defn}
\label{defn:separating_actions_rk}
 Let $t_1$ be a permuting $(n_1,0,k_1)$-action on $S_{g_1}$, and let $t_{2,j}$ be a permuting $(n_{2,j},0,k_{2,j})$-action on $\S_{g_{2,j}}(m_j)$, for $1 \leq j \leq s$. Then  $(\l t_1\r ,\l t_{2,1}\r ,\ldots,\l t_{2,s}\r )$ forms a \textit{(s+1)-compatible tuple  of degree n} if 
\begin{enumerate}[(i)]
\item for each $1\leq j\leq s$, $(\l t_1\r ,\l t_{2,j}\r )$ forms an $(u_j,v_j)$-compatible pair of degree $n$, for some $u_j, v_j \geq 0$ such that $k_{2,j} = u_j+v_j$.
\item For each $i\neq j$,
$$
\o(t_1,t_{2,i})\cap\o(t_1,t_{2,j}) = \emptyset = \p(t_1,t_{2,i})\cap\p(t_1,t_{2,j}).
$$

\end{enumerate}
The number $g = g_1+ \sum_{j=1}^s m_j(g_{2,j}+ k_{2,j} -1)$ is called the \textit{genus} of the $(s+1)$-tuple. 
\end{defn}

%\noindent The following lemma follows from Lemma~\ref{lem:single_orbit_actions}. 

\begin{defn}\label{defn:separating_multituples}
Fix $m,n \in \mathbb{N}$, and for $1 \leq i \leq m+1$, let 
$$\bar{t_i} = (\l t_{i,1}\r ,\l t_{i,2,1}\r ,\ldots,\l t_{i,2,s_i}\r )$$ be a $(s_i+1)$-compatible tuple as in Definition \ref{defn:separating_actions_rk}. Then the tuple
$$
( \overline{t_1} , \overline{t_2}, \ldots, \overline{t_{m+1}} )
$$
is said to form an \textit{(m+1)-compatible multituple of degree n} if for each $1\leq i\leq (m+1)$, 
\begin{enumerate}[(i)]
\item the pair $(\l t_{i,1}\r ,\l t_{i+1,1}\r )$ forms an $(r_{i,1},k_{i,1})$-compatible pair of degree $n$, and
\item
$
\o(t_{i,1},t_{i+1,1})\cap \left( \sqcup_{j=1}^{s_i} \o(t_{i,1},t_{i,2,j})\right) = \emptyset = \p(t_{i,1},t_{i+1,1})\cap \left( \sqcup_{j=1}^{s_i} \p(t_{i,1},t_{i,2,j})\right).
$
%\item $|\o_{t_{1,i}}| = \sum_{j=1}^{k_i}|\o_{t_{2,i,j}}| + |\o(t_{1,i},t_{1,i+1})|$.
\end{enumerate}
If $g(\overline{t_i})$ denotes the genus of $\overline{t_i}$, and
$$
\alpha_i := \sum_{\o \in\o(t_{1,i},t_{1,i+1})} |\o|
$$
then the number 
$$
g = \sum_{i=1}^{m+1} g(\overline{t_i}) + \sum_{i=1}^m (r_{1,i} + k_{1,i}\alpha_i-1)
$$
is called the \textit{genus} of the multituple. 
\end{defn}

\noindent The following theorem follows from Lemmas~\ref{thm:homologous_root_action} and  \ref{lem:single_orbit_actions}.

\begin{thm}\label{thm: compatible_multituple}
Let $\C$ be separating multicurve in $S_g$. Then the conjugacy class of a root $h$ of $t_{\C}$ of degree $n$ with $|\Fix_h(\C)| = m$ corresponds to an $(m+1)$-compatible  multituple of degree $n$ and genus $g$. 
\end{thm}

\noindent The following corollary, which gives an upper bound for the degree of a root of $t_{\C}$, follows from \cite[Theorem 8.6]{KR1} and Theorem~\ref{thm: compatible_multituple}.

\begin{cor}
\label{cor:bound_sep}
Let $\C$ be a separating multicurve in $S_g$ and $h$ be a root of $t_{\C}$ of degree $n$. Then $n \leq 4{g(\C)}^2 + 2g(\C)$.
\end{cor}

\noindent The following corollary, which gives a stable upper bound for the degree of a root of $t_{\C}$, follows from \cite[Theorem 8.14]{KR1} and Theorem~\ref{thm: compatible_multituple}. 

\begin{cor}   
\label{cor:stablebound_sep}
Let $\C$ be a separating multicurve in $S_g$ and $h$ be a root of $t_{\C}$ of degree $n$.  Then $n \leq \frac{16}{5}g^2 + 12g + \frac{45}{4}$, 
if $g(\C) \geq 10$. 
\end{cor}

The primary purpose behind developing the notion of data sets was for algebraically describing actions that are involved in the construction of roots. Owing to complexity of notation, we shall abstain from developing separate notation involving data sets for encoding compatible tuples or multituples. We will see later in Sections \ref{sec:classify_genus3} and \ref{sec:classify_genus4} that one can easily circumvent the need for complex notation by a judicious use of Remark \ref{rem: permuting_action_sgm}.

Recall that a multicurve $\C$ is said to be mixed if it is neither separating nor nonseparating. For the sake of brevity, we do not develop a separate theory to describe roots of Dehn twists about mixed multicurves as such a theory would merely be a combination of the separating and nonseparating cases. However, we will classify such roots on $S_3$ and $S_4$, thereby indicating how such a theory would  follow from the ideas developed in Sections \ref{sec:nonsepmulcurves} and \ref{sec:sepmulcurves}.

\section{Classification of roots for $S_3$}
\label{sec:classify_genus3}

In this, and the subsequent section, we classify roots of Dehn twists about multicurves on surfaces of genus 3 and 4. Once again, we break up the classification into nonseparating, separating and mixed multicurves.

When classifying an $(m+1)$-compatible multituples $(\overline{t_1},\overline{t_2},\ldots, \overline{t_{m+1}})$ that corresponds to a root, Condition (i) of Definition \ref{defn:separating_multituples} and Condition (ii) of Definition  \ref{defn:separating_actions_rk} help in eliminating data sets that do not lead to roots. For the sake of brevity, we only list those data sets that do lead to roots. Furthermore, in each case, a careful examination of the data set $\D$ also gives $\o_{\D}$, and so we only display the former. 

Finally, when $\overline{t_i}$ is a permuting $(n,r,k)$-action on $\S_g(m)$, we use Remark \ref{rem: permuting_action_sgm} and replace $\overline{t_i}$ by the corresponding action on $S_g$ which has an $m^{th}$ root $\widetilde{t_i}$, whose equivalence class can be encoded by a data set $D_i$. Therefore an $(m+1)$-compatible multituple is described by a tuple $(D_1,D_2,\ldots, D_{m+1})$ of data sets, which will be listed in a table. While enumerating the curves in a multicurve, we use the convention that separating curves are denoted with the letter $c$, while nonseparating curves are denoted with the letter $d$. \\

\noindent \textbf{$\C$ is a nonseparating multicurve.}

\begin{table}[H]
\begin{tabular}{|l|c|}
\hline
Degree & $D_1$ \\
\hline
3 & $(3,0;(1,3),(1,3))$ \\
\hline
\end{tabular}
\caption{$\C = \{d_1,d_2,d_3\}$ and $(r,k)=(0,1)$}
\end{table}

\begin{table}[H]
\begin{tabular}{|l|c|}
\hline
Degree & $D_1$  \\
\hline
2 & $(2,1;)$ \\
\hline
\end{tabular}
\caption{$\C = \{d_1,d_2\}$ and $(r,k)=(0,1)$}
\end{table}

\noindent{\textbf{$\C$ is a separating multicurve.}}
\begin{table}[H]
\begin{tabular}{|l|c|c|}
\hline
Degree & $D_1$ & $D_2$ \\
\hline
3 & $(3,0;(1,3),(1,3))$ & $(1,1;)$ \\
\hline
\end{tabular}
\caption{$\C = \{c_1,c_2,c_3\}, S_3 = S_0\#_{\C} S_1(3)$ and $(r,k)=(0,1)$}
\end{table}

\begin{table}[H]
\begin{tabular}{|l|c|c|c|}
\hline
Degree & $D_1$ & $D_2$ & $D_3$ \\
\hline
2 & $(1,1,1;)$ & $(2,0,1;(1,2))$ & $(1,1;)$ \\
\hline
\end{tabular}
\caption{$\C=\{c_1,c_2,c_3\}, S_3=S_0 \#_{c_1}S_1\#_{\{c_2,c_3\}}S_1(2)$ and $(r,k)=(1,1)$}
\end{table}

\begin{table}[H]
\begin{tabular}{|l|c|c|c|c|}
\hline
Degree & $D_1$ & $D_2$ & $D_3$ \\
\hline
3 & $(3,0,2;(2,3),(2,3))$ & $(3,0,(2,2);(2,3))$ & $(3,0,2;(2,3),(2,3))$ \\
6 & $(6,0,5;(1,2),(2,3))$ & $(3,0,(1,1);(1,3))$ & $(6,0,5;(1,2),(2,3))$ \\
6 & $(6,0,1;(1,2),(1,3))$ & $(1,1,(1,1);)$ & $(6,0,1;(1,2),(1,3))$ \\
6 & $(3,0,2;(2,3),(2,3))$ & $(2,0,(1,1);(1,2),(1,2))$ & $(3,0,2;(2,3),(2,3))$ \\
6 & $(2,0,1;(1,2),(1,2),(1,2))$ & $(3,0,(2,2);(2,3))$ & $(2,0,1;(1,2),(1,2),(1,2))$ \\
12 & $(3,0,1;(1,3),(1,3))$ & $(4,0,(3,3);(1,2))$ & $(3,0,1;(1,3),(1,3))$ \\
12 & $(4,0,3;(1,2),(3,4))$ & $(3,0,(1,1);(1,3))$ & $(4,0,3;(1,2),(3,4))$ \\
\hline
\end{tabular}
\caption{$\C = \{c_1,c_2\}, S_3=S_1\#_{c_1}S_1\#_{c_2}S_1$ and $(r,k)=(2,0)$}
\end{table}

\begin{table}[H]
\begin{tabular}{|l|c|c|}
\hline
Degree & $D_1$ & $D_2$ \\
\hline
2 & $(2,1;)$ & $(1,1;)$ \\
12 & $(4,0;(1,4),(1,2),(1,4))$ & $(6,0;(5,6),(1,2),(2,3))$ \\
12 & $(4,0;(3,4),(1,2),(3,4))$ & $(6,0;(5,6),(1,2),(2,3))$ \\
12 & $(6,0;(5,6),(1,2),(2,3))$ & $(4,0;(1,4),(1,2),(1,4))$ \\
12 & $(6,0;(5,6),(1,2),(2,3))$ & $(4,0;(3,4),(1,2),(3,4))$ \\
\hline
\end{tabular}
\caption{$\C = \{c_1,c_2\}, S_3 = S_1\#_{\C}S_1(2)$ and $(r,k)=(0,1)$}
\end{table}

\begin{table}[H]
\begin{tabular}{|l|c|c|c|}
\hline
Degree & $D_1$ & $D_2$ \\
\hline
2 & $(2,0,(1,1);(1,2),(1,2),(1,2))$ & $(1,0,(1,1);)$ \\
3 & $(3,0,(1,1);(1,3))$ & $(1,0,(1,1);)$ \\
3 & $(3,0,(2,2);(2,3))$ & $(3,0,(2,2);(2,3))$ \\
4 & $(4,0,(1,1);(1,2))$ & $(1,0,(1,1);)$ \\
4 & $(4,0,(3,3);(1,2))$ & $(2,0,(1,1);(1,2),(1,2),(1,2))$ \\
6 & $(3,0,(2,2);(2,3))$ & $(2,0,(1,1);(1,2),(1,2),(1,2))$ \\
12 & $(4,0,(3,3);(1,2))$ & $(3,0,(1,1);(1,3))$ \\
\hline
\end{tabular}
\caption{$\C = \C^{(2)}$, $S_3 = S_1 \#_{\C} S_1$, and $(r,k)=(2,0)$}
\end{table}

\begin{table}[H]
\begin{tabular}{|l|c|c|}
\hline
Degree & $D_1$ & $D_2$ \\
\hline
6 & $(2,0;(1,2),(1,2),(1,2),(1,2))$ & $(6,0;(1,6),(1,2),(1,3))$ \\
\hline
\end{tabular}
\caption{$\C = \C^{(2)},S_3 = S_1\#_{\C} S_1$ and $(r,k)=(0,1)$}
\end{table}

\noindent \textbf{$\C$ is a mixed multicurve.}

\begin{table}[H]
\begin{tabular}{|l|c|c|}
\hline
Degree & $D_1$ & $D_2$ \\
\hline
2 & $(3,0;(1,3),(1,3))$ & $(1,1;)$ \\
\hline
\end{tabular}
\caption{$\C=\{c_1,c_2,c_3,d_1,d_2,d_3\}, S_3=S_0\#_{\{c_1,c_2,c_3\}}\# S_1(3)$ where $d_i\in \widehat{S_1^i}$ and $(r,k)=(0,2)$}
\end{table}

\begin{table}[H]
\begin{tabular}{|l|c|c|c|}
\hline
Degree & $D_1$ & $D_2$ & $D_3$ \\
\hline
2  & $(1,1,1;)$ & $(2,0,1;(1,2))$ & $(1,1;)$ \\
\hline
\end{tabular}
\caption{$\C = \{c_1,c_2,c_3,d_1,d_2\}$, $S_3=S_1\#_{c_1}S_0\#_{\{c_2,c_3\}}S_1(2)$ where $d_i \in \widehat{S_1^i}$ for $i=1,2$ and $(r,k)=(1,2)$}
\end{table}

\begin{table}[H]
\begin{tabular}{|l|c|c|}
\hline
Degree & $D_1$ & $D_2$  \\
\hline
2 & $(2,1;)$ & $(2,0;)$ \\
\hline
\end{tabular}
\caption{$\C=\{c_1,c_2,d_1,d_2\}, S_3=S_1\#_{\{c_1,c_2\}}S_1(2)$ and $d_i \in \widehat{S_1^i}$ for $i=1,2$ and $(r,k)=(0,2)$}
\end{table}

\begin{table}[H]
\begin{tabular}{|l|c|c|}
\hline
Degree & $D_1$ & $D_2$  \\
\hline
2 & $(2,0,1;(1,2),(1,2),(1,2))$ & $(1,1,1;)$ \\
\hline
\end{tabular}
\caption{$\C=\{c_1,d_1,d_2\}, S_3=S_1\#_{c_1}S_2$ and $\{d_1,d_2\} \subset \widehat{S_2}$ and $(r,k)=(1,1)$}
\end{table}

\begin{table}[H]
\begin{tabular}{|l|c|c|c|}
\hline
Degree & $D_1$ & $D_2$ \\
\hline
3 & $(3,0,2;(2,3),(2,3))$ & $(3,0,(2,2,2);)$ \\
\hline
\end{tabular}
\caption{$\C = \{c_1,d_1\}$ mixed and $S_3 = S_1\#_{c_1}S_2$ where $d_1 \in \widehat{S_2}$ and $(r,k)=(2,0)$}
\end{table}

\section{Classification of roots for $S_4$}
\label{sec:classify_genus4}

\noindent\textbf{$\C$ is a nonseparating multicurve.} 

\begin{table}[H]
\begin{tabular}{|l|c|}
\hline
Degree & $D_1$ \\
\hline
4 & $(4,0;(1,4),(1,4))$ \\
2 & $(2,1;(1,2),(1,2))$ \\
\hline
\end{tabular}
\caption{$\C= \{d_1,d_2,d_3,d_4\}$ and $(r,k)=(0,1)$}
\end{table}

\begin{table}[H]
\begin{tabular}{|l|c|}
\hline
Degree & $D_1$ \\
\hline
3 & $(3,1;)$ \\
\hline
\end{tabular}
\caption{$\C= \{d_1,d_2,d_3\}$ and $(r,k)=(0,1)$}
\end{table}

\begin{table}[H]
\begin{tabular}{|l|c|}
\hline
Degree & $D_1$ \\
\hline
2 & $(2,1;(1,2),(1,2))$ \\
\hline
\end{tabular}
\caption{$\C = \{d_1,d_2\}$ and $(r,k)=(0,1)$}
\end{table}

%------------------------------
%------------------------------

\noindent\textbf{$\C$ is a separating multicurve.} 

\begin{table}[H]
\begin{tabular}{|l|c|c|}
\hline
Degree & $D_1$ & $D_2$ \\
\hline
4 & $(4,0;(1,4),(1,4))$ & $(1,1;)$ \\
\hline
\end{tabular}
\caption{$\C = \{c_1,c_2,c_3,c_4\}, S_4 = S_0\#_{\C} S_1(4)$ and $(r,k)=(0,1)$}
\end{table}

\begin{table}[H]
\begin{tabular}{|l|c|c|c|}
\hline
Degree & $D_1$ & $D_2$ & $D_3$ \\
\hline
2 & $(1,1;)$ & $(2,0;)$ & $(1,1;)$ \\
6 & $(6,0;(1,6),(1,2),(1,3))$ & $(2,0;)$ & $(6,0;(1,6),(1,2),(1,3))$ \\
\hline
\end{tabular}
\caption{$\C = \{c_1,c_2,c_3,c_4\},S_4 =  S_1(2)\#_{\{c_1,c_2\}}S_0 \#_{\{c_3,c_4\}}S_1(2)$ and $(r,k)=(0,2)$}
\end{table}

\begin{table}[H]
\begin{tabular}{|l|c|c|}
\hline
Degree & $D_1$ & $D_2$ \\
\hline
3 & $(3,1;)$ & $(1,1;)$ \\
6 & $(6,0; (1,6),(1,2),(1,3))$ & $(1,1;)$\\
6 & $(6,0; (5,6),(1,2),(2,3))$ & $(1,1;)$ \\
6 & $(3,1;)$ & $(6,0; (1,6),(1,2),(1,3))$ \\
6 & $(3,1;)$ & $(6,0; (5,6),(1,2),(2,3))$ \\
\hline
\end{tabular}
\caption{$\C=\{c_1,c_2,c_3\}, S_4 = S_1\#_{\C} S_1(3)$ and $(r,k)=(0,1)$.}
\end{table}

\begin{table}[H]
\begin{tabular}{|l|c|c|c|}
\hline
Degree & $D_1$ & $D_2$ & $D_3$ \\
\hline
4 & $(1,1,1;)$ & $(4,0,1;(1,2),(1,4))$ & $(1,1;)$ \\
6 & $(1,1,1;)$ & $(6,0,1;(1,2),(1,3))$ & $(1,1;)$ \\
6 & $(3,0,2;(2,3),(2,3))$ & $(2,0,1;(1,2),(1,2),(1,2))$ & $(6,0;(1,6),(1,2),(1,3))$ \\
\hline
\end{tabular}
\caption{$\C = \{c_1,c_2,c_3\}, S_4 = S_1 \#_{c_1}S_1\#_{\{c_2,c_3\}}S_1(2)$ and $(r,k)=(1,1)$}
\end{table}

\begin{table}[H]
\begin{tabular}{|l|c|c|c|c|}
\hline
Degree & D1 & D2 & D3 &  D4 \\
\hline
3 & $(3,0,2;(2,3),(2,3))$ & $(3,0,2;(2,3),(2,3))$ & $(3,0,2;(2,3),(2,3))$ & $(3,0,2;(2,3),(2,3))$ \\
6 & $(6,0,5;(1,2),(2,3))$ & $(3,0,1;(1,3),(1,3))$ & $(6,0,5;(1,2),(2,3))$ & $(3,0,1;(1,3),(1,3))$ \\
12 & $(3,0,1;(1,3),(1,3))$ & $(4,0,3;(1,2),(3,4))$ & $(3,0,1;(1,3),(1,3))$ & $(4,0,3;(1,2),(3,4))$ \\
12 & $(4,0,3;(1,2),(3,4))$ & $(3,0,1;(1,3),(1,3))$ & $(4,0,3;(1,2),(3,4))$ & $(3,0,1;(1,3),(1,3))$ \\
\hline
\end{tabular}
\caption{$\C = \{c_1,c_2,c_3\}, S_4 = S_1\#_{c_1}S_1\#_{c_2}S_1\#_{c_3}S_1$ and $(r,k)=(3,0)$} 
\end{table}

\begin{table}[H]
\begin{tabular}{|l|c|c|}
\hline
Degree & $D_1$ & $D_2$ \\
\hline
6 & $(2,0;(1,2),(1,2),(1,2),(1,2),(1,2),(1,2))$ & $(6,0;(1,6),(1,2),(1,3))$ \\
6 & $(2,1;(1,2),(1,2))$ & $(6,0;(1,6),(1,2),(1,3))$ \\
6 & $(6,0;(5,6),(1,3),(5,6))$ & $(1,1;)$ \\
\hline
\end{tabular}
\caption{$\C= \{c_1,c_2\}, S_4 = S_2 \#_{\C}S_1(2)$ and $(r,k)=(0,1)$.}
\end{table}

\begin{table}[H]
\begin{tabular}{|l|c|c|c|}
\hline
Degree & $D_1$ & $D_2$ & $D_3$ \\
\hline
6 & $(6,0,5;(1,2),(2,3))$ & $(3,0,1;(1,3),(1,3))$ & $(6,0,5;(1,3),(5,6))$ \\
\hline
\end{tabular}
\caption{$\C = \{c_1,c_2\}, S_4 = S_1\#_{c_1}S_1\#_{c_2}S_2$ and $(r,k)=(2,0)$.}
\end{table}

\begin{table}[H]
\begin{tabular}{|l|c|c|}
\hline
Degree & $D_1$ & $D_2$\\
\hline
3 & $(3,0;(1,3),(1,3),(1,3))$ & $(3,0;(1,3),(1,3),(1,3))$ \\
3 & $(3,0;(1,3),(1,3),(1,3))$ & $(3,0;(2,3),(2,3),(2,3))$ \\
3 & $(3,0;(2,3),(2,3),(2,3))$ & $(3,0;(2,3),(2,3),(2,3))$ \\
3 & $(3,0;(2,3),(2,3),(2,3))$ & $(3,0;(1,3),(1,3),(1,3))$ \\
6 & $(3,0;(1,3),(1,3),(1,3))$ & $(6,0;(1,6),(1,2),(1,3))$ \\
6 & $(3,0;(1,3),(1,3),(1,3))$ & $(6,0;(5,6),(1,2),(2,3))$ \\
6 & $(3,0;(2,3),(2,3),(2,3))$ & $(6,0;(1,6),(1,2),(1,3))$ \\
6 & $(3,0;(2,3),(2,3),(2,3))$ & $(6,0;(5,6),(1,2),(2,3))$ \\
\hline
\end{tabular}
\caption{$\C=\C^{(3)}, S_4= S_1\#_{\C}S_1$ and $(r,k)=(0,1)$}
\end{table}

\begin{table}[H]
\begin{tabular}{|l|c|c|}
\hline
Degree & $D_1$ & $D_2$\\
\hline
4 & $(2,0,1;(1,2),(1,2),(1,2))$ & $(4,0,3;(1,2),(3,4))$ \\
\hline
\end{tabular}
\caption{$\C=\C^{(3)}, S_4 = S_1\#_{\C}S_1$ and $(r,k)=(1,1)$}
\end{table}
\begin{table}[H]
\begin{tabular}{|l|c|c|}
\hline
Degree & $D_1$ & $D_2$ \\
\hline
3 & $(3,0,(2,2,2);)$ & $(3,0,(2,2,2);)$ \\
6 & $(3,0,(2,2,2);)$ & $(2,0,(1,1,1);(1,2))$ \\
\hline
\end{tabular}
\caption{$\C = \C^{(3)}, S_4 = S_1\#_{\C} S_1$ and $(r,k)=(3,0)$.}
\end{table}

\begin{table}[H]
\begin{tabular}{|l|c|c|c|}
\hline
Degree & $D_1$ & $D_2$ & $D_3$ \\
\hline
4 & $(4,0;(1,4),(1,2),(1,4))$ & $(2,0;(1,2),(1,2))$ & $(4,0;(1,4),(1,2),(1,4))$ \\
4 & $(4,0;(1,4),(1,2),(1,4))$ & $(2,0;(1,2),(1,2))$ & $(4,0;(3,4),(1,2),(3,4))$ \\
4 & $(4,0;(3,4),(1,2),(3,4))$ & $(2,0;(1,2),(1,2))$ & $(4,0;(3,4),(1,2),(3,4))$ \\
6 & $(6,0;(1,6),(1,2),(1,3))$ & $(2,0;(1,2),(1,2))$ & $(6,0;(1,6),(1,2),(1,3))$ \\
6 & $(6,0;(1,6),(1,2),(1,3))$ & $(2,0;(1,2),(1,2))$ & $(6,0;(5,6),(1,2),(2,3))$ \\ 
6 & $(6,0;(5,6),(1,2),(2,3))$ & $(2,0;(1,2),(1,2))$ & $(6,0;(5,6),(1,2),(2,3))$ \\
\hline
\end{tabular}
\caption{$\C = \C^{(2)}\sqcup \C^{(2)}, S_4 = S_1\#_{\C^{(2)}} S_0 \#_{\C^{(2)}} S_1$ and $(r,k)=(0,2)$.}
\end{table}

\begin{table}[H]
\begin{tabular}{|l|c|c|}
\hline
Degree & $D_1$ & $D_2$ \\
\hline
2 & $(2,0;(1,2),(1,2))$ & $(1,1;)$ \\
6 & $(2,0;(1,2),(1,2))$ & $(6,0;(1,6),(1,2),(1,3))$ \\
\hline
\end{tabular}
\caption{$\C = \C^{(2)}\sqcup \C^{(2)}, S_4 = S_0\#_{\C} S_1(2)$ and $(r,k)=(0,2)$.}
\end{table}

\begin{table}[H]
\begin{tabular}{|l|c|c|}
\hline
Degree & $D_1$ & $D_2$ \\
\hline
4 & $(2,0;(1,2),(1,2),(1,2),(1,2),(1,2),(1,2))$ & $(4,0;(1,4),(1,2),(1,4))$ \\
4 & $(2,0;(1,2),(1,2),(1,2),(1,2),(1,2),(1,2))$ & $(4,0;(3,4),(1,2),(3,4))$ \\
4 & $(2,1;(1,2),(1,2))$ & $(4,0;(1,4),(1,2),(1,4))$ \\
4 & $(2,1;(1,2),(1,2))$ & $(4,0;(3,4),(1,2),(3,4))$ \\
4 & $(4,0;(1,2),(1,2),(3,4))$ &  $(2,0;(1,2),(1,2),(1,2),(1,2))$ \\
4 & $(4,0;(1,2),(1,2),(3,4))$ & $(2,1;)$ \\
4 & $(4,0;(1,2),(1,2),(2,4))$ & $(2,0;(1,2),(1,2),(1,2),(1,2))$ \\
4 & $(4,0;(1,2),(1,2),(2,4))$ & $(2,1;)$ \\
6 & $(2,0;(1,2),(1,2),(1,2),(1,2),(1,2),(1,2))$ & $(6,0;(1,6),(1,2),(1,3))$ \\
6 & $(2,1;(1,2),(1,2))$ & $(6,0;(1,6),(1,2),(1,3))$ \\
6 & $(6,0;(5,6),(1,3),(5,6))$ & $(2,0;(1,2),(1,2),(1,2),(1,2))$ \\
6 & $(6,0;(5,6),(1,3),(5,6))$ & $(2,1;)$ \\
\hline
\end{tabular}
\caption{$\C = \C^{(2)}, S_4 = S_2\#_{\C} S_1$ and $(r,k)=(0,1)$.}
\end{table}

\begin{table} [H]
\begin{tabular}{|l|l|l|}
\hline
Degree & $D_1$ & $D_2$ \\
\hline
2 & $(1,2,(1,1);)$ & $(2,0,(1,1);(1,2),(1,2))$ \\
2 & $(2,0,(1,1);(1,2),(1,2),(1,2),(1,2))$ & $(1,1,(1,1);)$ \\
2 & $(2,1,(1,1);)$ & $(1,1,(1,1);)$ \\
3 & $(1,2,(1,1);)$ & $(3,0,(1,1);(1,3))$ \\
3 & $(3,0,(1,1);(2,3),(2,3))$ & $(1,1,(1,1);)$ \\
3 & $(3,0,(2,2);(1,3),(1,3))$ & $(3,0,(2,2);(1,3),(1,3))$ \\
4 & $(1,2,(1,1);)$ & $(4,0,(1,1);(1,2))$ \\
5 & $(5,0,(1,1);(3,5))$ & $(1,1,(1,1);)$ \\
6 & $(6,0,(1,1);(2,3))$ & $(1,1,(1,1);)$ \\
6 & $(2,0,(1,1);(1,2),(1,2),(1,2),(1,2))$ & $(3,0,(2,2);(2,3))$ \\
6 & $(2,0,(1,1);(1,2),(1,2),(1,2),(1,2))$ & $(3,0,(2,2);(2,3))$ \\
6 & $(2,1,(1,1);)$ & $(3,0,(2,2);(2,3))$ \\
6 & $(3,0,(2,2);(1,3),(1,3))$ & $(2,0,(1,1);(1,2),(1,2))$ \\
6 & $(6,0,(5,5);(1,3)$ & $(3,0,(1,1);(1,3))$ \\
12 & $(3,0,(1,1);(2,3),(2,3))$ & $(4,0,(3,3);(1,2))$ \\
12 & $(6,0,(5,5);(1,3))$ & $(4,0,(1,1);(1,2))$ \\
15 & $(5,0,(3,3);(4,5))$ & $(3,0,(2,2);(2,3))$ \\
20 & $(5,0,(4,4);(2,5))$ & $(4,0,(1,1);(1,2))$ \\
\hline
\end{tabular}
\caption{$\C = \C^{(2)}, S_4 = S_2 \#_{\C} S_1$ and $(r,k)=(2,0)$} 
\end{table}

\begin{table}[H]
\begin{tabular}{|l|c|c|c|}
\hline
Degree & $D_1$ & $D_2$ & $D_3$\\
\hline
6 & $(3,0,2;(2,3),(2,3))$ & $(2,0,1;(1,2),(1,2),(1,2))$ & $(6,0;(1,6),(1,2),(1,3))$ \\
6 & $(1,1,1;)$ & $(6,0,1;(1,2),(1,3))$ & $(2,0;(1,2),(1,2),(1,2),(1,2))$ \\
\hline
\end{tabular}
\caption{$\C = \C^{(2)}\sqcup\{c_1\}, S_4 = S_1\#_{c_1} S_1\#_{\C^{(2)}} S_1$ and $(r,k)=(1,1)$.}
\end{table}
\begin{table}[H]
\begin{tabular}{|l|c|c|c|c|}
\hline
Degree & $D_1$ & $D_2$ & $D_3$ \\
\hline
3 & $(3,0,2;(2,3),(2,3))$ & $(3,0,(2,2,2);)$ & $(3,0,(2,2);(2,3))$ \\
12 & $(4,0,3;(1,2),(3,4))$ & $(3,0,(1,1,1);)$ & $(4,0,(3,3);(1,2))$ \\
\hline
\end{tabular}
\caption{$\C = \{c_1\}\sqcup \C^{(2)}, S_4  = S_1 \#_{c_1} S_1 \#_{\C^{(2)}} S_1$ and $(r,k)=(3,0)$.}
\end{table}

%------------------------------

\noindent\textbf{$\C$ is a mixed multicurve.} 

\begin{table}[H]
\begin{tabular}{|l|c|c|}
\hline
Degree & $D_1$ & $D_2$ \\
\hline
4 & $(4,0;(1,4),(1,4)$ & $(1,0;)$ \\
\hline
\end{tabular}
\caption{$\C = \{c_1,\ldots, c_4, d_1,\ldots, d_4\}$, $S_4 = S_0\#_{\{c_1,\ldots, c_4\}} S_1(4)$ with $d_i \in \widehat{S_1^i}$ and $(r,k)=(0,2)$.}
\end{table}

\begin{table}[H]
\begin{tabular}{|l|c|c|c|}
\hline
Degree & $D_1$ & $D_2$ & $D_3$ \\
\hline
2 & $(1,0;)$ & $(2,0;(1,2),(1,2))$ & $(1,0;)$ \\
\hline
\end{tabular}
\caption{$\C = \{c_1,\ldots, c_4, d_1,\ldots, d_4\}$, $S_4 = S_1(2)\#_{\{c_1,c_2\}}S_0\#_{\{c_3,c_4\}}S_1(2)$ with $d_i \in \widehat{S_1^i}$ and $(r,k)=(0,4)$.}
\end{table}

\begin{table}[H]
\begin{tabular}{|l|c|c|c|}
\hline
Degree & $D_1$ & $D_2$ & $D_3$ \\
\hline
2 & $(1,1;)$ & $(2,0;(1,2),(1,2))$ & $(1,0;)$ \\
\hline
\end{tabular}
\caption{$\C=\{c_1,\ldots, c_4,d_1,d_2\}, S_4 = S_1(2)\#_{\{c_1,c_2\}}S_0\#_{\{c_3,c_4\}}S_1(2)$ with $d_i \in \widehat{S_1^i}$ and $(r,k)=(0,3)$}
\end{table}

\begin{table}[H]
\begin{tabular}{|l|c|c|c|c|}
\hline
Degree & $D_1$ & $D_2$ & $D_3$ & $D_4$ \\
\hline
2 &  $(1,0;)$ & $(1,1,1;)$ & $(2,0,1;(1,2))$ & $(1,1,1;)$  \\
\hline
\end{tabular}
\caption{$\C=\{c_1,\ldots, c_4, d_1,d_2\}, S_4 = S_1(2)\#_{\{c_1,c_2\}}\left(S_1\#_{c_3}S_0\#_{c_4}S_1\right)$ with $d_i \in \widehat{S_1^i}$ and $(r,k)=(2,2)$}
\end{table}

\begin{table}[H]
\begin{tabular}{|l|c|c|}
\hline
Degree & $D_1$ & $D_2$\\
\hline
3 & $(3,1;)$ & $(1,0;)$ \\
\hline
\end{tabular}
\caption{$\C=\{c_1,c_2,c_3,d_1,d_2,d_3\}, S_4 = S_1\#_{\{c_1,c_2,c_3\}}S_1(3)$ with $d_i \in \widehat{S_1^i}$ and $(r,k)=(0,2)$}
\end{table}

\begin{table}[H]
\begin{tabular}{|l|c|c|c|}
\hline
Degree & $D_1$ & $D_2$ & $D_3$ \\
\hline
2 & $(2,0,1;(1,2),(1,2),(1,2))$ & $(1,1,1;)$ & $(4,0;(1,4),(1,4))$ \\
\hline
\end{tabular}
\caption{$\C=\{c_1,c_2,c_3,d_1,d_2\}, S_4=S_1\#_{c_1}S_1\#_{\{c_2,c_3\}}S_1(2)$ with $d_i \in \widehat{S_1^i}$ and $(r,k)=(1,2)$}
\end{table}

\begin{table}[H]
\begin{tabular}{|l|c|c|}
\hline
Degree & $D_1$ & $D_2$\\
\hline
2 & $(2,0;(1,2),(1,2),(1,2),(1,2),(1,2),(1,2))$ & $(4,0;(1,4),(1,4))$ \\
2 & $(2,1;(1,2),(1,2))$ & $(4,0;(1,4),(1,4))$ \\
\hline
\end{tabular}
\caption{$\C=\{c_1,c_2,d_1,d_2\}, S_4 = S_2\#_{\{c_1,c_2\}}S_1(2)$ with $d_i \in \widehat{S_1^i}$ and $(r,k)=(0,2)$}
\end{table}

\begin{table}[H]
\begin{tabular}{|l|c|c|}
\hline
Degree & $D_1$ & $D_2$\\
\hline
2 & $(2,0;(1,2),(1,2))$ & $(4,0;(1,4),(1,2),(1,4))$ \\
2 & $(2,0;(1,2),(1,2))$ & $(4,0;(3,4),(1,2),(3,4))$ \\
\hline
\end{tabular}
\caption{$\C=\{c_1,c_2,d_1,d_2\}, S_4=S_2\#_{\{c_1,c_2\}}S_1(2)$ with $d_i \in \widehat{S_2}$ and $(r,k)=(0,2)$}
\end{table}

\begin{table}[H]
\begin{tabular}{|l|c|c|}
\hline
Degree & $D_1$ & $D_2$\\
\hline
2 & $(1,1,1;)$ & $(2,0,1;(1,2),(1,2),(1,2))$ \\
\hline
\end{tabular}
\caption{$\C=\{c_1,d_1,d_2\}, S_4 = S_1\#_{c_1}S_3$ with $\{d_1,d_2\}\subset \widehat{S_3}$ and $(r,k)=(1,1)$}
\end{table}

\begin{table}[H]
\begin{tabular}{|l|c|c|}
\hline
Degree & $D_1$ & $D_2$\\
\hline
2 & $(2,0;(1,2),(1,2))$ & $(2,0;(1,2),(1,2),(1,2),(1,2))$ \\
\hline
\end{tabular}
\caption{$\C = \C^{(2)}\sqcup\{d_1,d_2\},S_4 = S_2\#_{\C^{(2)}}S_1$ with $\{d_1,d_2\} \subset \widehat{S_2}$ and $(r,k)=(0,2)$.}
\end{table}

\section{Concluding remarks}

\subsection{Roots and the Torelli Group}
Let  $\Psi : \Mod(S_g) \to \Sp(2g,\mathbb{Z})$ be the symplectic representation of $\Mod(S_g)$ arising out of its action on $H_1(S_g,\mathbb{Z})$, and let $\I(S_g)$ denote the kernel of $\Psi$, the Torelli group of $S_g$. If $\C$ contains nonseparating curves, then $\Psi(t_{\C})$ is a product of commuting elementary matrices, and thus a root of $t_{\C}$ induces a root of such a matrix in $Sp(2g,\mathbb{Z})$. However, if every curve in $\C$ is separating, then $\Psi(h)$ is a root of unity in $Sp(2g,\mathbb{Z})$. Using the theory we have developed for multicurves, we shall now give a succinct proof of the fact that such a root cannot lie in $\I(S_g)$.

\begin{thm}\label{thm:no_roots_Tor}
Let $h$ be the root of the Dehn twist $t_{\C}$ about a multicurve $\C$ in $S_g$. Then $h \notin \I(S_g)$. 
\end{thm}
\begin{proof}
Since $\I(S_g)$ is normal in $\Mod(S_g)$, it suffices to prove that the conjugacy class of $h$ intersects $\I(S_g)$ non-trivially. If $c$ is an essential simple closed curve in $S_g$, it follows from \cite[Proposition 6.3, Example 6.5.2]{MF} that $t_c \in \I(S_g)$ if and only if $c$ is a separating curve. Consequently, it suffices to assume that every curve in $\C$ is a separating curve. 

In that case, let $\Fix_h(\C) = \{\E_1,\ldots,\E_m\}$, where $m = |\Fix_h(\C)|$, and write
$$
S_g = \csum_{i=1}^{m(h)}  (S_{g_i} \#_{\E_i} S_{g_{i+1}}),
$$
as in Notation~\ref{not:rk_permuting_root}. By Theorem~\ref{thm: compatible_multituple}, the conjugacy class of $h$ corresponds to an $(m+1)$-compatible multituple $( \overline{t_1} , \overline{t_2}, \ldots, \overline{t_{m+1}} )$, where 
$\bar{t_i} = (\l t_{i,1}\r ,\l t_{i,2,1}\r ,\ldots,\l t_{i,2,s_i}\r )$ is a $(s_i+1)$-compatible tuple of genus $g_i$.

Choose $1\leq i\leq m$ such that at least one component of $\overline{t_i}$ is non-trivial, so we write
$$
S_{g_i} = \csum_{j=1}^{k_i} (S_{g_{i,1}} \#_{D_{i,j}} \S_{g_{i,2,j}}(m_{1,j})),
$$
as in Notation \ref{not:rk_permuting_root}. If $g_{i,1} \geq 1$, and $t_{i,1}$ is a non-trivial finite order element in $\Mod(S_{g_{i,1}})$, then $t_{i,1} \notin \I(S_{g_{i,1}})$ by \cite[Theorem 6.8]{MF}. Since $t_{i,1}$ is obtained from a restriction of $h$ by coning, it follows that $h\notin \I(S_g)$. If $t_{i,1}$ is trivial or $g_{i,1}=0$, then it follows $k_i \geq 1$ and that $t_{i,2,j}$ is non-trivial for some $1\leq j\leq k_i$. Since each curve in $\C$ is separating, $g_{i,2,j} \geq 1$, and so the surface $\S_{g_{i,2,j}}(m_{i,j})$ contributes $2m_{i,j}g_{i,2,j}$ generators to the standard geometric symplectic basis of $H_1(S_g,\mathbb{Z})$. Now $t_{i,2,j}$ (and consequently $h$) cyclically permutes the components of $\S_{g_{i,2,j}}(m_{i,j})$, and thus $h\notin \I(S_g)$.
%If $S_{g_i} \cap \C = \emptyset$ for all $i$, then there exists 
%$s \in \{1,\ldots,m\}$ such that $t_{1,s}$ is a nontrival permuting $(n_{1,s},0,k_{2,s})$-action on $S_{g_{1,s}}$. Then it follows from \cite[Theorem 6.8]{MF} that $t_{1,s} \notin \I(S_{g_1})$, and hence $h \notin \I(S_g)$. If $\C \cap S_{g_{b}} \neq \emptyset$ for some $b \in \{1,\ldots,m\}$, then $S_{g_b}(D_b)$(where $D_b = \C \cap S_{g_b}$ contains a surface $\S_{g_{2,b,j}}(m_{b,j})$ for some some $j$ and some $m_{b,j}>1$. Moreover, the $m_{b,j}$ components of $\S_{g_{2,b,j}}(m_{b,j})$ are cyclically permuted by $t_{2,b,j}$. Since these 
%components contribute $2m_{b,j}g_{2,b,j}$ generators to the standard generating set of $H_1(S_g,\mathbb{Z})$ which are permuted by $t_{2,b,j}$, we have that $h \notin \I(S_g)$. 

\end{proof}

\subsection{Roots of finite product of powers}
The theory developed in this paper for classifying roots up to conjugacy for finite products of commuting Dehn twists can be naturally generalised to one that classifies roots of finite products of powers of commuting twists. 

Currently, the compatibility condition requires that pairs of distinguished orbits (or fixed points) of permuting actions should have associated angles that add up to $2\pi/n \pmod{2\pi}$. When $c$ is a single nonseparating curve, the roots of $t_c^{\ell}$ for $1 \leq \ell < n$, were classified in \cite{KR2} by using a variant of this condition, which required that the angles associated with compatible fixed points add up  to $2\pi\ell/n \pmod{2\pi}$. This notion of compatibility of fixed points can be generalized to orbits, and this will eventually lead to the classification of roots of homeomorphisms of the form $\prod_{i=1}^m t_{c_i}^{\ell_i}$, where $\{c_1,c_2,\ldots,c_m\}$ is a multicurve. It is also apparent that such roots would not lie in $\I(S_g)$ for the same reasons as above.
%----------------------------------------------------

\section{Acknowledgements}

The authors would like to thank Dan Margalit and Darryl McCullough for helpful suggestions and examples that made the theory more comprehensive.

\bibliographystyle{plain}
\bibliography{ultimate_version}

\end{document}